\documentclass{article}
\usepackage{latexsym,amsmath,amssymb,amsthm}
\usepackage{algorithm,algpseudocode,color}
\usepackage{graphicx,makecell,multirow,url}
\usepackage{placeins}
\newtheorem{theorem}{Theorem}[section]

\newtheorem{definition}{Definition}[section]

\newtheorem{lemma}{Lemma}[section]

\newtheorem{proposition}{Proposition}[section]

\newtheorem{assumption}{Assumption}[section]
\newtheorem{corollary}[theorem]{Corollary}
\newtheorem{example}{Example}
\def\oR{\Bar{\R}}
\def\disp{\displaystyle}

\def\tto{\;{\lower 1pt \hbox{$\rightarrow$}}\kern-10pt
\hbox{\raise 2pt \hbox{$\rightarrow$}}\;}

\def\Bar{\overline}

\def\ve{\varepsilon}
\def\epsilon{\varepsilon}

\def\R{\Bbb R}

\def\N{\Bbb N}
\def\ox{\bar{x}}

\def\gph{\mbox{\rm gph}\,}

\def\dom{\mbox{\rm dom}\,}

\def\ph{\varphi}

\def\oR{\Bar{\R}}

\def\dd{\delta}

\def\O{\Omega}
\def\ph{\varphi}

\def\dd{\delta}

\newcommand{\revise}[1]{{\color{black}#1}}

\title{A Globally Convergent Proximal Newton-Type Method\\ in Nonsmooth Convex Optimization}

\author{Boris S. Mordukhovich \thanks{Corresponding author. Department of Mathematics, Wayne State University, Detroit, Michigan, USA ({\tt
boris@math.wayne.edu}). Research of this author was partly supported by the USA National Science Foundation under grants DMS-1512846 and DMS-1808978, by the USA Air Force Office of Scientific Research
grant \#15RT04, and by Australian Research Council under grant DP-190100555.}\and  Xiaoming Yuan\thanks{Department of Mathematics, The University of Hong Kong, Hong Kong, China ({\tt xmyuan@hku.hk}).
This author was supported by the General Research Fund 12302318 from the Hong Kong Research Grants Council.}\and Shangzhi Zeng\thanks{Department of Mathematics and Statistics, University of Victoria, Canada. ({\tt zengshangzhi@uvic.ca}). Research of this author was supported by the Pacific Institute for the Mathematical Sciences (PIMS).} \and Jin Zhang\thanks{Corresponding author. Department of Mathematics,
Southern University of Science and Technology, National Center for Applied Mathematics Shenzhen, Shenzhen, 518055, China ({\tt zhangj9@sustech.edu.cn}). Research of this author was supported by National Science Foundation of China 11971220, by Shenzhen
Science and Technology Program (No. RCYX20200714114700072), and by  the Stable Support Plan Program of Shenzhen Natural Science Fund (No. 20200925152128002)}}

\begin{document}

\maketitle

\noindent{\bf Abstract.} The paper proposes and justifies a new
algorithm of the proximal Newton type to solve a broad class of
nonsmooth composite convex optimization problems without strong
convexity assumptions. Based on advanced notions and techniques of
variational analysis, we establish implementable results on the
global convergence of the proposed algorithm as well as its local
convergence with superlinear and quadratic rates. For certain
structured problems, the obtained local convergence conditions do
not require the local Lipschitz continuity of the corresponding
Hessian mappings that is a crucial assumption used in the literature
to ensure a superlinear convergence of other algorithms of the
proximal Newton type. The conducted numerical experiments of solving
the $l_1$
regularized logistic regression model illustrate the possibility of applying the proposed algorithm to deal with practically important problems.
\vspace*{-0.15in}

\section{Introduction}\vspace*{-0.05in}

In this paper we consider a class of optimization problems of the
following type: \begin{equation}\label{prob_ori_1}
\min_{x\in\mathbb{R}^n}~F(x):=f(x)+g(x), \end{equation} where both
functions $f,g\colon\R^n\to\oR:=(-\infty,\infty]$ are proper,
convex, and lower semicontinuous (l.s.c.), while being structurally
different from each other. Namely, $f$ is assumed to be twice
continuously differentiable with the Lipschitz continuous gradient
$\nabla f$ on its domain. On the other hand, $g$ is merely
continuous on its domain; see Assumption~\ref{assumptions_f} for the
precise formulations. It has been well recognized that model
\eqref{prob_ori_1}, known as a {\em composite convex optimization
problem}, frequently appears in a variety of applications including,
e.g., machine learning, signal processing, and statistics, where $f$
is a {\em loss function} and $g$ is a {\em regularizer}; we keep
this terminology here. Note that problem \eqref{prob_ori_1} contains
in fact implicit constraints written as $x\in\O:=\dom g$.

It is typical in applications that problems of type
\eqref{prob_ori_1} have a large size, which makes attractive to
compute their solutions by employing first-order algorithms such as
the {\em proximal gradient method} (PGM). Given each iterate $x^k$,
the PGM constructs a new $x^{k+1}$ by solving the following
optimization subproblem, which approximates the smooth function $f$
in \eqref{prob_ori_1} by the linear model:
\begin{equation}\label{lin} \min_{x\in\mathbb{R}^n}~l_k(x) +
\frac{1}{2t}\|x-x_k\|^2 \;\; \text{with} \;\; l_k(x):=f(x^k)+\nabla
f(x^k)^T(x-x^k)+g(x), \end{equation} where $^T$ indicates the matrix
transposition, and where $t > 0$ represents the step size of PGM. As
well known, the PGM applied to \eqref{prob_ori_1} generates a
sequence of iterates that converges at least sublinearly of rate
$O({1}/{k})$ (see, e.g., \cite{FISTA,Nesterov13_convex}) and
linearly with respect to the sequence of cost function
values---provided that $f$ is strongly convex; see e.g.,
\cite{Schmidt2011convergence}. Refined results on linear convergence
of the PGM are derived under various error bound conditions as in
\cite{LP-KL2017,LT-1992,FDM-ref1,Tseng2010approximation,FBconvex}.

When $f$ is a twice continuously differentiable function, it is
natural to expect algorithms having faster convergence rates by
exploiting the Hessian $\nabla^2f(x^k)$ of $f$ at each iterate $x^k$
and constructing the next iterate $x^{k+1}$ as a solution to the
following quadratic subproblem: \begin{equation}\label{subproblem}
\min_{x\in\mathbb{R}^n}~q_k(x):=f(x^k)+\nabla
f(x^k)^T(x-x^k)+\frac{1}{2}(x-x^k)^TH_k(x-x^k)+g(x), \end{equation}
where $H_k$ is an appropriate approximation of the Hessian
$\nabla^2f(x^k)$. Methods of this type to solve composite
optimization problems \eqref{prob_ori_1} are unified under the name
of {\em proximal Newton-type methods}; see, e.g., \cite{Lee2014}. To
the best of our knowledge, the origin of such methods to solve
nonsmooth composite optimization problems given in form
\eqref{prob_ori_1} can be traced back to the generalized proximal
point method developed by Fukushima and Mine \cite{Fukushima1981}
who in turn considered it as an extension of Rockafellar's proximal
point method \cite{roc} to find zeros of maximal monotone operators
and subgradient inclusions associated with convex functions. On the
other hand, the general scheme of {\em successive quadratic
approximations} to solve optimization-related problems is a common
idea of Newton-type and quasi-Newton methods; see the books
\cite{fp,Solo14} with their bibliographies. For particular
subclasses of composite problems \eqref{prob_ori_1}, the quadratic
approximation scheme \eqref{subproblem} contains special versions of
the proximal Newton-type methods known as GLMNET
\cite{Friedman2007}, newGLMNET \cite{Yuan2012}, QUIC
\cite{Hsieh2011}, the Newton-LASSO method \cite{Oztoprak2012}, the
projected Newton-type algorithms
\cite{Schmidt2011convergence,Schmidt2011}, etc.

Observe further that, due to the convexity of both functions $f$ and
$g$ with $f$ being smooth, problem \eqref{prob_ori_1} can be
equivalently written as the {\em generalized equation}
\begin{equation}\label{GE} 0\in\nabla f(x)+\partial g(x)
\end{equation} in the sense of Robinson \cite{rob}, where $\partial
g(x)$ is the subdifferential of $g$ at $x$, and where the used
subdifferential sum rule does not require any qualification
conditions due to the smoothness of $f$; see, e.g.,
\cite[Proposition~1.30]{m18}. Then subproblem \eqref{subproblem} for
constructing the new iterate $x^{k+1}$ in the proximal Newton method
for \eqref{GE} reduces to solving the following {\em partially
linearized} generalized equation at the given iterate $x^k$:
\begin{equation}\label{linearized_GE} 0\in\nabla
f(x^k)+H_k(x-x^k)+\partial g(x). \end{equation} Various results on
the local superlinear and quadratic convergence of iterative
sequences $\{x^k\}$ for \eqref{linearized_GE} are obtained in the
literature in the framework of quasi-Newton methods for generalized
equations under different kinds of regularity conditions imposed on
$\partial F$ from \eqref{prob_ori_1}; see, e.g., the books
\cite{Dontchev09,fp,Solo14} with the references and discussions
therein. In particular, Fischer \cite{fischer2002local} proposes an
iterative procedure to solve generalized equations and proves local
superlinear and quadratic convergence of iterates under a certain
Lipschitz stability property of the corresponding perturbed solution
map. More specifically, paper \cite{fischer2002local} develops a
quasi-Newton algorithm to solve \eqref{prob_ori_1} in the framework
of \eqref{linearized_GE} that exhibits a local superlinear/quadratic
convergence in the setting where $g$ is the indicator function of a
box constraint, and where $H_k$ in \eqref{subproblem} is taken as
the regularized Hessian $H_k:=\nabla^2 f(x^k)+\alpha_k I$ with
$\{\alpha_k\}$ being a positive vanishing sequence satisfying
certain conditions. The main assumptions of \cite{fischer2002local}
include the local Lipschitz continuity of the Hessian $\nabla^2f(x)$
and the upper Lipschitz continuity/calmness of the perturbed
solution map \eqref{prob_ori_1} at the points in question.

However, how to build a reasonable {\em globalization} of the local scheme given by \eqref{subproblem} has not been completely resolved yet. Various globalizations of the proximal Newton method can be
found in the literature, see, e.g., \cite{byrd2016inexact,Lee2014,Lee2019,Scheinberg2016}. Unfortunately, all these works require $f$ to be {\em strongly convex}. In particular, paper by Byrd et al.
\cite{byrd2016inexact}, which addresses the special case of problem \eqref{prob_ori_1} with $g:=\lambda\|x\|_1$ and $\lambda>0$, proposes implementable inexactness conditions and backtracking line search
procedures to design a globally convergent proximal Newton method, but the local superlinear and quadratic convergence  results therein are established under the strong convexity assumption on $f$. Quite
recently, in \cite{yue2019family}, the inexactness conditions and backtracking line search procedures of \cite{byrd2016inexact} is applied to develop a proximal Newton method for \eqref{prob_ori_1} with
proving its local convergence of superlinear and quadratic rates by using the Luo-Tseng error bound condition \cite{LT-1992} instead of the strong convexity assumption in \cite{byrd2016inexact}. However,
the convergence results in \cite{yue2019family} have a crucial flaw. To achieve a local quadratic convergence rate, the authors of \cite{yue2019family} require that parameters of their method satisfy a
certain condition involving the constant in the error bound, which is extremely challenging to estimate. Note to this end that the strong
convexity assumption has not been imposed by using some other Newton-type algorithms such as the one based on the forward-backward envelope (FBE), which is different from the proximal Newton-type method
developed below; see, e.g., \cite{TSP18} and the references therein.\vspace*{0.03in}

In this paper we design a new globally convergent proximal Newton-type algorithm to solve composite convex optimization problems of class \eqref{prob_ori_1} under the following {\em standing assumptions}
on the given data without requiring the strong convexity of the loss function $f$:

\begin{assumption}\label{assumptions_f} Impose the following properties of the loss function and the regularizer in \eqref{prob_ori_1}:
\begin{itemize}
\item[{\bf(i)}] Both functions $f,g:{\R^n}\rightarrow (-\infty,\infty]$ are proper, l.s.c., and convex.

\item[{\bf(ii)}] The \revise{effective} domain of the loss function $\dom f:=\{x~|~f(x)<\infty\}$ is open, and $f(x)$ is twice continuously differentiable on a closed set $\Omega\supset \dom f$.

\item[{\bf(iii)}] The regularizer $g(x)$ is continuous on its domain and $\varnothing \ne\dom g \subset\dom f$.

\item[{\bf(iv)}] The gradient $\nabla f(x)$ is Lipschitz continuous on a closed set $\Omega$ from {\rm(ii)} with Lipschitz constant $L_1>0$.

\item[{\bf(v)}] Problem \eqref{prob_ori_1} has a nonempty solution set ${\cal{X}^*}:=\arg\min_{x\in\R^n}F(x)$ with the optimal value $F^*$.
\end{itemize}
\end{assumption}

Basic convex analysis tells us that the imposed assumptions (ii) and
(iii) ensure the fulfillment of the subdifferential sum rule
$\partial F(x)=\nabla f(x)+\partial g(x)$ for all $x\in\dom g$; see,
e.g., \cite[Corollary~2.45]{mn}.\vspace*{0.05in}

Our {\em main contributions} can be summarized as follows:

\begin{itemize}
\item[\bf(1)]
We develop a {\em globally convergent} proximal Newton-type algorithm to solve \eqref{prob_ori_1} with an {\em implementable inexact condition} for subproblem \eqref{subproblem} and a new reasonable
{\em backtracking line search} strategy. Our line search procedure does not require any restrictive assumptions. It is shown in this way that if the subgradient
mapping $\partial F$ is {\em metrically subregular} at some limiting point of the iterative sequence, the backtracking line search procedure accepts a unit step size when the iterates are close to the
solution. Furthermore, we prove that the proposed proximal Newton-type algorithm exhibits a {\em local convergence} with the {\em quadratic convergence rate}. Numerical experiments are performed to solve
the {\em $l_1$ regularized logistic regression} problem that illustrate the efficiency of the proposed algorithm.

\item[\bf(2)] We establish novel local convergence results for the proposed algorithm under the {\em metric $q$-subregularity} assumption imposed on the subgradient mapping $\partial F$ for any
positive number $q >\frac{1}{2}$. If $q<1$, the obtained results require {\em less restrictive assumptions} in comparison with the case of metric subregularity $(q=1)$ to ensure a superlinear
convergence of iterates, while for $q>1$ we achieve a convergence rate that is {\em higher than quadratic}.

\item[\bf(3)]
When the loss function $f$ in \eqref{prob_ori_1} satisfies additional structural assumptions, we obtain a local superlinear convergence rate of our proposed algorithm {\em without imposing the Lipschitz
continuity} of the Hessian matrix $\nabla^2 f(x)$. The latter assumption is crucial for establishing a fast convergence of the previously known algorithms of the proximal Newton type.
\end{itemize}

The rest of the paper is organized as follows. Section~\ref{sec:2} briefly overviews the notions and results of variational analysis needed for the subsequent material. 
An concrete example is also given to justify our motivation of using metric subregularity instead of the Luo-Tseng error bound condition.
In Section~\ref{sec:3} we present
our proximal Newton-type algorithm and establish its global convergence. In Section~\ref{sec:5}, for the cases where $q\in(0,1]$ and $q>1$, we separately derive local fast convergence results under
the metric $q$-subregularity of $\partial F$. Specially, local superlinear and quadratic convergence of the proposed algorithm under metric subregularity are given. Section~\ref{sec:6} is devoted to
problem \eqref{prob_ori_1} with a certain structure of the loss function $f$ and establishes in this case a superlinear convergence of the proposed algorithm without the Lipschitz continuity of the loss
function Hessian. Finally, Section~\ref{sec:7} conducts and analyzes numerical experiments to solve the practically important $l_1$ regularized logistic regression problem by implementing the designed
proximal Newton-type method.\vspace*{-0.15in}

\section{Preliminaries from Variational Analysis}\label{sec:2}\vspace*{-0.05in}

Here we recall and discuss some material from variational analysis that is broadly used in what follows. The reader can find more details and references in the books
\cite{Dontchev09,m18,Rockafellar2009variational}.

Throughout the paper, we use the standard notation. Recall that $\R^n$ signifies an $n$-dimensional Euclidean space with the inner product $\langle\cdot,\cdot\rangle$ and the norm denoted by
$\|\cdot\|$, while the $1$-norm is signified by $\|\cdot\|_1$. For any matrix $A\in\R^{m\times n}$ we have $\|A\|:=\max_{x \ne 0}\frac{\|Ax\|}{\|x\|}$ with $\tilde{\sigma}_{\min}(A)$ standing for the
smallest nonzero singular value of $A$. The symbols $\mathbb{B}_{r}(x)$ and $\overline{\mathbb{B}}_{r}(x)$ denote the open and the closed Euclidean norm ball centered at $x$ with radius $r>0$,
respectively, while we use $\mathbb{B}$ and $\overline{\mathbb{B}}$ for the corresponding unit balls around the origin. Given a nonempty subset $\O\subset\R^n$, denote by ${\rm bd}\,\O$ its boundary and
consider the associated distance function ${\rm dist}(x;\O):=\inf\{\|x-y\|\,\big|\,y\in\O\}$ and the indicator function $\delta_{\O}(x)$ equal 0 if $x\in\O$ and $\infty$ otherwise. The graph of a
set-valued mapping/multifunction $\Psi\colon\R^n\rightrightarrows\R^m$ is given by
$\gph\Psi:=\{(x,\upsilon)\in\R^n\times\R^m~|~\upsilon\in\Psi(x)\}$, and the inverse to $\Psi$ is $\Psi^{-1}(\upsilon):=\{x\in\R^n~|~\upsilon\in\Psi(x)\}$.\vspace*{0.03in}

The following fundamental properties of set-valued mappings are employed in the paper to establish fast local convergence results for the proposed proximal Newton-type algorithm.

\begin{definition}\label{subreg} Let $\Psi\colon\R^n\tto\R^m$ be a set-valued mapping, let $(\ox,\bar{\upsilon})\in\gph\Psi$, and let $q>0$.
\begin{itemize}
\item[{\bf(i)}] We say that $\Psi$ is {\sc metrically $q$-subregular} at $(\ox,\bar{\upsilon})$ with modulus $\kappa>0$ if there is $\epsilon>0$ such that
\begin{equation}\label{subreg1}
{\rm dist}\big(x;\Psi^{-1}(\bar\upsilon)\big)\le\kappa\,{\rm dist}\big(\bar\upsilon;\Psi(x)\big)^q\;\mbox{ for all }\;x\in\mathbb{B}_{\epsilon}(\ox).
\end{equation}

\item[{\bf(ii)}] $\Psi$ is said to be {\sc metrically subregular} at $(\ox,\bar{\upsilon})$ if $q=1$ in \eqref{subreg1}.

\item[{\bf(iii)}] We say that $\Psi$ is {\sc strongly metrically $q$-subregular} at $(\ox,\bar{\upsilon})$ with modulus $\kappa>0$ if there exists $\epsilon>0$ such that
\begin{equation*}
\|x - \bar{x}\|\le\kappa\,{\rm dist}\big(\bar\upsilon;\Psi(x)\big)^q\;\mbox{ for all }\;x\in\mathbb{B}_{\epsilon}(\ox).
\end{equation*}
\end{itemize}
\end{definition}

The metric subregularity property has been well recognized and applied in variational analysis and optimization numerical aspects. The reader can find more information and references in
\cite{Dontchev09,m18} with the commentaries and the bibliographies therein. In this paper we employ metric subregularity of {\em subgradient mappings}, which form a remarkable class of multifunctions with
special properties. Various sufficient conditions and characterizations of this property of subgradient mappings are given in \cite{a,ag,dmn} in terms of certain second-order growth conditions imposed on
the function in question.

The metric $q$-subregularity of order $q\in(0,1)$, known also as {\em H\"older metric subregularity}, is much less investigated, while some verifiable conditions for the fulfillment of this property can
be found in, e.g., \cite{Gaydu2011,Li2012,zn}. Note that the H\"older metric subregularity is clearly a weaker assumption in comparison with the standard metric subregularity property.

The case of {\em higher-order metric subregularity} with $q>1$ in \eqref{subreg1} is largely open in the literature. One of the reasons for this is that the corresponding {\em metric $q$-regularity}
property with $q>1$ does not make sense, since it holds only for constant mappings. Nevertheless, it is shown in \cite{mordukhovich2015higher} that the higher-order metric subregularity is a useful
property in variational analysis and optimization. This property is characterized for subgradient mappings in \cite{mordukhovich2015higher} via a higher-order growth condition, and its {\em strong}
version from Definition~\ref{subreg}(iii) is applied therein to the convergence analysis of quasi-Newton methods for generalized equations.\vspace*{0.03in}

Next we consider the {\em proximal mapping}
\begin{equation}\label{prox-map}
\hbox{Prox}_{g}(u):={\rm argmin}\Big\{g(x)+\frac{1}{2}\|x-u\|^2\;\Big|\;x\in\R^n\Big\},\quad u\in\R^n,
\end{equation}
associated with a proper function $g\colon\R^n\to\oR$. A crucial role of proximal mappings has been well recognized not only in proximal Newton-type algorithms (see, e.g., \cite{byrd2016inexact,Lee2014}),
but also in other second-order methods of numerical optimization. In particular, we refer the reader to the very recent papers \cite{kmp,BorisEbrahim}, where the proximal mappings are used for designing
superlinearly convergent Newton-type algorithms to find tilt-stable local minimizers of nonconvex extended-real-valued functions and to solve subgradient inclusions in a large generality. If $g$ is
l.s.c.\  and convex, then the proximal mapping \eqref{prox-map} is single-valued and nonexpansive on $\R^n$, i.e., Lipschitz continuous with constant one; see, e.g.,
\cite[Theorem~12.12]{Rockafellar2009variational}.

It is important to emphasize that in many practical models of type \eqref{prob_ori_1} arising, in particular, in machine learning and statistics, the proximal mapping associated with the regularizer
term $g$ (e.g., when $g$ is the $l_1$-norm, the group Lasso regularizer, etc.) can be easily computed. This is the case of the $l_1$ regularized logistic regression problem in our applications developed
in Section~\ref{sec:7}.\vspace*{0.03in}

Having \eqref{prox-map}, define further the {\em prox-gradient mapping} associated with \eqref{prob_ori_1} by
\begin{equation}\label{prox-gra}
\mathcal{G}(x):=x-\mathrm{Prox}_{g}\big(x-\nabla f(x)\big),\quad x\in\R^n,
\end{equation}
and present some properties of \eqref{prox-gra} used in what follows. Note that $\mathcal{G}(x)$ is generally defined in terms of a positive parameter $L > 0$ as
$\mathcal{G}_L(x):= x-\mathrm{Prox}_{\frac{1}{L}g}\big(x-\frac{1}{L}\nabla f(x)\big)$. In order to concentrate on the main idea, we simply set $L = 1$ throughout this paper. All of the results in this
paper can be easily extended to the case with a given positive $L$. Thanks to the convexity, we have that $\mathcal{G}(x)=0$ if and only if $x\in\mathcal{X}^*$. The following proposition provides an
upper estimate for $\|\mathcal{G}(x)\|$ by $\mathrm{dist}(0;\nabla f(x)+\partial g(x))$ and shows that $\mathcal{G}(x)$ is Lipschitz continuous. It can be seen as a direct combination of
\cite[Theorem~3.5]{DL2016} and \cite[ Lemma~10.10]{FOM}.

\begin{proposition}\label{bound_rxk0}
Let $\nabla f$ be Lipschitz continuous with modulus $L_1$ on $\R^n$. Then we have the estimates
\begin{equation*}
\|\mathcal{G}(x)\| \le \mathrm{dist}(0;\nabla f(x)+\partial g(x)) \;\mbox{ for all }\;x \in \dom f,
\end{equation*}
\begin{equation*}
\|\mathcal{G}(x) - \mathcal{G}(y)\|\le(2+L_1)\|x - y\| \;\mbox{ for any }\;x, y \in \dom f.
\end{equation*}
\end{proposition}

The next proposition is a combination of \cite[Theorems~3.4 and 3.5]{DL2016}.

\begin{proposition}\label{bound_dist} Let $\nabla f$ be Lipschitz continuous with modulus $L_1$ around $\ox$, and let the mapping $\nabla f(x)+\partial g(x)$ be metrically subregular at $(\bar{x},0)$,
i.e., there exist numbers $\epsilon,\kappa>0$ such that
\[
\mathrm{dist}(x;\mathcal{X}^*)\le\kappa\,\mathrm{dist}\big(0;\nabla f(x)+\partial g(x)\big)\;\mbox{ for all }\;x\in\mathbb{B}_{\epsilon}(\bar{x}).
\]
Then whenever $x\in\mathbb{B}_{\epsilon}(\bar{x})$ we have the estimate
\[
\mathrm{dist}(x;\mathcal{X}^*)\le(1+\kappa)(1+L_1)\|\mathcal{G}(x)\|.
\]
\end{proposition}

The following proposition gives a reverse statement to Proposition~\ref{bound_dist} while providing an estimate of the norm of \eqref{prox-gra} via the distance to the solution set of the convex
composite problem \eqref{prob_ori_1}.

\begin{proposition}\label{bound_rxk} Let $\nabla f$ be Lipschitz continuous with modulus $L_1$ on $\R^n$. Then we have the estimate
\begin{equation*}
\|\mathcal{G}(x)\|\le(2+L_1)\mathrm{dist}(x;\mathcal{X}^*)\;\mbox{ for all }\;x\in \dom f.
\end{equation*}
\end{proposition}
\begin{proof}
Observe first that the mapping $\mathcal{G}(x)$ is well-defined and single-valued for all $x\in\dom f$ due to the aforementioned result of \cite{Rockafellar2009variational}. It easily follows
from Assumption~\ref{assumptions_f} that the nonempty solution set ${\cal X}^*$ is closed and convex; hence each point $x\in\R^n$ has the unique projection $\pi_x\in\mathcal{X}^*$ onto ${\cal X}^*$.
Note that $\mathcal{G}(\pi_x) = \pi_x-\mathrm{Prox}_{g}(\pi_x-\nabla f(\pi_x))=0$ for $\pi_x\in\mathcal{X}^*$. Thus we verify the claim of the proposition by
\[
\|\mathcal{G}(x)\|= \|\mathcal{G}(x) - \mathcal{G}(\pi_x)\| \le(2+L_1)\|x-\pi_x\|,\quad x\in \dom f,
\]
where the inequality holds since $\mathcal{G}(x)$ is $(2+L_1)$-Lipschitz continuous by Proposition \ref{bound_rxk0}.
\end{proof}

Next we obtain an extension of Proposition~\ref{bound_dist} to case where the subgradient mapping $\nabla f+\partial g$ in \eqref{prob_ori_1} satisfies the H\"older subregularity property in the point in
question.

\begin{proposition}\label{lem3} Let $\nabla f$ be Lipschitz continuous with modulus $L_1$ around $\ox$, and let the mapping $\nabla f(x)+\partial g(x)$ be metrically $q$-subregular at $(\bar{x},0)$ with
$q\in(0,1]$, i.e., there exist $\epsilon_1,\kappa_1>0$ such that
\[
\mathrm{dist}(x;\mathcal{X}^*)\le\kappa_1\mathrm{dist}\big(0;\nabla f(x)+\partial g(x)\big)^q\;\mbox{ for all }\;x\in\mathbb{B}_{\epsilon_1}(\bar{x}).
\]
Then we find constants $\epsilon_2,\kappa_2>0$ that ensure the estimate
\begin{equation}\label{lem3a}
\mathrm{dist}(x;\mathcal{X}^*)\le\kappa_2\|\mathcal{G}(x)\|^q\;\mbox{ whenever }\;x\in\mathbb{B}_{\epsilon_2}(\bar{x}).
\end{equation}
\end{proposition}
\begin{proof} By \eqref{prox-gra} we have the inclusions
\[
\mathcal{G}(x)\in\nabla f(x)+\partial g\big(x-\mathcal{G}(x)\big)\;\mbox{ and}
\]
\[
\mathcal{G}(x)+\nabla f\big(x - \mathcal{G}(x)\big)-\nabla f(x)\in\nabla f\big(x-\mathcal{G}(x)\big)+\partial g\big(x-\mathcal{G}(x)\big)
\]
for all $x\in\dom f$. When $x\in\mathbb{B}_{\epsilon_1}(\bar{x}) \cap \dom f$, it follows from the imposed assumption that
\[
\begin{aligned}
\mathrm{dist}\big(x-\mathcal{G}(x);\mathcal{X}^*\big)&\le\kappa_1\,\mathrm{dist}\big(0;\mathcal{G}(x)+\nabla f\big(x - \mathcal{G}(x)\big)-\nabla f(x)\big)^q\\
&\le\kappa_1(1+L_1)^q\|\mathcal{G}(x)\|^q,
\end{aligned}
\]
which leads us to the resulting estimates for such $x$:
\[
\begin{aligned}
\mathrm{dist}(x;\mathcal{X}^*)&\le\mathrm{dist}\big(x-\mathcal{G}(x);\mathcal{X}^*\big)+\|\mathcal{G}(x)\|\\
&\le (1+\kappa_1(1+L_1)^q)\max\big\{\|\mathcal{G}(x)\|,\|\mathcal{G}(x)\|^q\big\}.
\end{aligned}
\]
Applying now Proposition~\ref{bound_rxk} tells us that, whenever $\mathrm{dist}(x;\mathcal{X}^*)\le 1/(2+L_1)$ and $x\in \dom f$, we get
\begin{equation*}
\|\mathcal{G}(x)\|\le(2+L_1)\mathrm{dist}(x;\mathcal{X}^*)\le 1.
\end{equation*}
Letting $\epsilon_2:=\min\{1/(2+L_1),\epsilon_1\}$ and remembering that $q \le 1$ bring us to the inequality
\begin{equation*}
\mathrm{dist}(x;\mathcal{X}^*)\le(1+\kappa_1(1+L_1)^q)\|\mathcal{G}(x)\|^q\;\mbox{ for all }\;x\in\mathbb{B}_{\epsilon_2}(\bar{x}),
\end{equation*}
which verifies \eqref{lem3a} with $\kappa_2:=(1+\kappa_1(1+L_1)^q)$ and thus completes the proof of the proposition.
\end{proof}

Finally, we establish a sufficient condition for the metric $q$-subregularity of the subgradient mapping $\nabla f(x)+\partial g(x)$. Recall first the following characterization of metric
$q$-subregularity, which is a direct specification of \cite[Theorem~3.4]{mordukhovich2015higher} in the convex case.

\begin{proposition}\label{char_q}
Let $\bar{x} \in \dom F$, $\bar{\upsilon} \in \partial F(\bar{x})$, and $q>0$. Then we have the equivalent statements:
\begin{itemize}
\item[\bf(i)] $\partial F$ is metrically $q$-subregular at $(\ox,\bar{\upsilon})$.
\item[\bf(ii)] There are two positive numbers $\epsilon$ and $c$ such that
\[
F(x)\ge F(\bar{x}) + \langle \bar{\upsilon},x-\bar{x}\rangle + c\cdot\mathrm{dist}\big(x;(\partial F)^{-1}(\bar{\upsilon})\big)^{\frac{1+q}{q}}\;\mbox{ for all }\;x\in\mathbb{B}_{\epsilon}(\ox).
\]
\end{itemize}
\end{proposition}

The aforementioned sufficient conditions for the metric $q$-subregularity of the subgradient mapping $\nabla f(x)+\partial g(x)$ as formulated as follows.
\begin{proposition}\label{q-suff}
Let $\bar{x} \in \mathcal{X}^*$, and let $q > 0$. Suppose $\partial g$ is strongly metrically $q$-subregular at $(\bar{x},-\nabla f(\bar{x}))$. Then $\nabla f(x)+\partial g(x)$ is metrically
$q$-subregular at $(\bar{x},0)$.
\end{proposition}
\begin{proof}
Since $\partial g$ is strongly metrically $q$-subregular at $(\bar{x},-\nabla f(\bar{x}))$, we have $(\partial g)^{-1}(-\nabla f(\bar{x})) = \{\bar{x}\}$, and Proposition \ref{char_q} gives us positive
numbers $\epsilon$ and $c$ such that
\[
g(x) \ge g(\bar{x}) + \langle -\nabla f(\bar{x}),x-\bar{x} \rangle + c\| x - \bar{x} \|^{\frac{1+q}{q}}\;\mbox{ for all }\;x\in\mathbb{B}_{\epsilon}(\ox).
\]
The convexity of $f$ implies that
\[
f(x) \ge f(\bar{x}) + \langle \nabla f(\bar{x}),x-\bar{x} \rangle\;\mbox{ for all }\;x \in \mathbb{R}^n.
\]
Summing up the above two inequalities gives us
\[
F(x) \ge F(\bar{x}) + c\| x - \bar{x} \|^{\frac{1+q}{q}} \ge F(\bar{x}) + c\cdot\mathrm{dist}(x;\mathcal{X}^*)^{\frac{1+q}{q}} \;\mbox{ for all }\;x\in\mathbb{B}_{\epsilon}(\ox).
\]
Then the conclusion of the proposition immediately follows from Proposition~\ref{char_q}.
\end{proof}

{\color{black}
Metric subregularity is a weaker assumption than the Luo-Tseng error bound condition. We conclude this section by giving a specific example verifying this statement.

\begin{example}
{\rm Consider the following problem of composite convex optimization:
$$
min_{x \in \mathbb{R}^n} f(x) + g(x)\;\mbox{ with }\;f(x):= c^Tx\;\mbox{ and }\;g(x):=\|x\|,
$$
where $c \in \mathbb{R}^n$ is such that$\|c\| = 1$. Problems of this type frequently appear in statistical learning models. It can be easily calculated that the optimal value is $F^* = 0$ and the optimal solution is  $\mathcal{X}^* = \{ - \gamma c ~|~ \gamma \ge 0\}$. We know that for any $\bar{x} \in \mathcal{X}^*$ the mapping
$$
\nabla f(x) + \partial g(x) = c + \partial \|x\|
$$ 
is metrically subregular at $(\bar{x},0)$ since $\partial \|x\|$ is metrically subregular on its graph; see, e.g., \cite[Lemma~4]{FBconvex}. 

On the other hand, the Luo-Tseng error bound fails. Indeed,  if this condition holds, then there exist $\kappa, \epsilon, \eta > 0$ such that for any $x$ satisfying $F(x) \le \eta$ and $\|\mathcal{G}(x)\| \le \epsilon$ we get
\[
\mathrm{dist}(x;\mathcal{X}^*)  \le \kappa \|\mathcal{G}(x)\|.
\]
Let $c_k\to c$ with $\|c_k\| = 1$, $k\in\N$. Setting $x_k: = \gamma_k(-c_k)$ with $\gamma_k:= \frac{1}{\sqrt{-c^Tc_k + 1}} \rightarrow\infty$ gives us 
$$
F(x_k) = \gamma_k(-c^Tc_k + \|c_k\|) = \sqrt{-c^Tc_k + 1} \rightarrow 0,$$
$$
c -c_k= c + \frac{x_k}{\|x_k\|} \in \nabla f(x_k)+\partial g(x_k).
$$ 
It follows from Proposition~\ref{bound_rxk0} that we have the estimate
$$
\|\mathcal{G}(x_k)\| \le \mathrm{dist}\big(0;\nabla f(x)+\partial g(x)\big) \le \|c - c_k\| \rightarrow 0.
$$
However, letting $\theta_k := \arccos( \frac{c^T(c - c_k)}{\|c - c_k\|}) \rightarrow \pi/2$ tells us that
\[
\frac{\mathrm{dist}(x_k;\mathcal{X}^*) }{\|\mathcal{G}(x_k)\|} \ge \frac{\gamma_k\sqrt{1 - (c^Tc_k)^2}}{\|c - c_k\|} = \gamma_k \sin(\theta_k) \rightarrow \infty,
\]
which clearly contradicts the Luo-Tseng error bound condition.}
\end{example}}

\section{The New Algorithm and Its Global Convergence}\label{sec:3}\vspace*{-0.05in}

In this section we describe the proposed proximal Newton-type algorithm to solve the class of composite convex optimization problems \eqref{prob_ori_1} with justifying its global convergence under the
standing assumptions.

Given a current iterate $x^k$ for each $k=0,1,\ldots$, we select a {\em positive semidefinite matrix} $B_k$ as an arbitrary approximation of the Hessian $\nabla^2 f(x^k)$ satisfying the {\em standing
boundedness assumption}:
\begin{equation}\label{assum_B_bounded}
\mbox{ there exists }\;M\ge 0\;\mbox{ such that }\;\|B_k\|\le M\;\mbox{ whenever }\;k=0,1,\ldots.
\end{equation}
If the gradient mapping $\nabla f$ is uniformly Lipschitz continuous along the sequence of iterates with constant $L_1$, then \eqref{assum_B_bounded} holds for $B_k=\nabla^2 f(x^k)$ with $M=L_1$.
In the general case of $B_k$, pick any constants $c>0$ and $\rho\in(0,1]$ and, using the prox-regular mapping \eqref{prox-gra}, consider the positive number $\alpha_k:=c\|\mathcal{G}(x^k)\|^{\rho}$
and define the {\em quasi-Newton approximation} of the Hessian of $f$ at $x^k$ by
\begin{equation}\label{alg-app}
H_k:=B_k+\alpha_kI\;\mbox{ for all }\;k=0,1,\ldots,
\end{equation}
which is a positive definite matrix. Then similarly to \cite{byrd2016inexact}, but with the different approximation \eqref{alg-app}, denote
\begin{equation}\label{opt-mes}
r_k(x):=x-\mathrm{Prox}_{g}\left(x-\nabla f(x^k)-H_k(x-x^k)\right)
\end{equation}
and select $\hat{x}^k$ as an {\em approximate minimizer} of the quadratic {\em subproblem} for \eqref{prob_ori_1} given by
\begin{equation}\label{subproblem1}
\min_{x\in\mathbb{R}^n}~q_k(x):=f(x^k)+\nabla f(x^k)^T(x-x^k)+\frac{1}{2}(x-x^k)H_k(x-x^k)+g(x)
\end{equation}
with the {\em residual} number $\|r_k(\hat x^k)\|$ measuring the approximate optimality of $\hat x^k$ in \eqref{subproblem1}. Observing that $\|r_k(\hat{x}^k)\|=0$ if and only if $\hat{x}^k$ is an
{\em exact} solution to subproblem \eqref{subproblem1}, we use the nonnegative number {\color{black} $\|r_k(\hat{x}^k)\|$}
with $r_k(x)$ taken from \eqref{opt-mes} as the {\em optimality measure} of $\hat x^k$ in subproblem \eqref{subproblem1}. Adapting the scheme of \cite{byrd2016inexact} in our new setting, let us impose
the following two estimates as {\em inexact conditions} for choosing $\hat{x}^k$ as an approximate solution to subproblem \eqref{subproblem1}:
{\color{black}\begin{equation}\label{inexact_condition_1}
\|r_k(\hat{x}^k)\|\le\eta_k\|\mathcal{G}(x^k)\|\;\mbox{ and }\;q_k(\hat{x}^k)\le q_k(x^k)
\end{equation}}%
with the parameter $\eta_k:=\nu\min\{1,\|\mathcal{G}(x^k)\|^{\varrho}\}$ defined via \eqref{opt-mes} and some numbers $\nu\in[0,1)$ and $\varrho>0$. \revise{Since any exact solution to subproblem \eqref{subproblem1} always fulfills the two inexact conditions in \eqref{inexact_condition_1}, a point $\hat x^k$ satisfying \eqref{inexact_condition_1} always exists. For many application problems, the mapping $\mathrm{Prox}_{g}$ is easy to calculate, thus inexact conditions \eqref{inexact_condition_1} can be readily verified.
In the case where $\mathrm{dist}\big(0;\partial q_k(\hat{x}^k)\big)$ is easy to estimate, because $\|r_k(\hat{x}^k)\| \le \mathrm{dist}\big(0;\partial q_k(\hat{x}^k)\big)$ always holds (see, e.g., \cite[Theorem~3.5]{DL2016}),  we can use $\mathrm{dist}\big(0;\partial q_k(\hat{x}^k)\big)\le\eta_k\|\mathcal{G}(x^k)\|$ for verifying the first inexact condition in \eqref{inexact_condition_1}.}\vspace*{0.03in}

Using the above constructions and the line search procedure inspired by \cite{chen1999proximal,Dan2002}, we are ready to propose the proximal Newton-type algorithm designed as in Algorithm \ref{Newton}.
\begin{algorithm}
\caption{Proximal Newton-type method}\label{Newton}
\begin{algorithmic}[1]
\State Choose $x^0\in\mathbb{R}^n$, $0<\theta,\sigma,\gamma<1$, $C>F(x^0)$, $\revise{\bar{\alpha}}, c>0$, and $\rho\in(0,1]$.
\For{$k=0,1,\ldots$}{
\begin{itemize}
\item[1.] Update the approximation of the Hessian matrix $B_k$.
\item[2.] Form the quadratic model \eqref{subproblem} with $H_k:=B_k+\alpha_kI$ and \revise{$\alpha_k:=\min \left\{ \bar{\alpha}, c\|\mathcal{G}(x^k)\|^{\rho}\right\} $  }.
\item[3.] Obtain an inexact solution $\hat{x}^k$ of \eqref{subproblem} satisfying the conditions in \eqref{inexact_condition_1}.
\item[4.] If $k=0$, let $\vartheta_1:=\mathcal{G}(x^0)$ and go to Step~5. For $k\ge 1$, if $\|\mathcal{G}(\hat{x}^k)\|\le\sigma\vartheta_k$ and $F(\hat{x}^k)\le C$, let $t_k:=1$, $\vartheta_{k+1}:=
\|\mathcal{G}(\hat{x}^k)\|$,
and go to Step~6. Otherwise, let $\vartheta_{k+1}:=\vartheta_k$ and go to Step~5.
\item[5.] Perform a backtracking line search along the direction $d^k:=\hat{x}^k-x^k$ by setting $t_k:=\gamma^{m_k}$, where $m_k$ is the smallest nonnegative integer $m$ such that
\begin{equation}\label{line_search}
F(x^k+\gamma^md^k)\le F(x^k)-\theta\alpha_k\gamma^m\|d^k\|^2.
\end{equation}
\item[6.] Set $x^{k+1}:=x^k+t_kd^k$.
\end{itemize}}
\EndFor
\end{algorithmic}
\end{algorithm}
In the proposed Algorithm \ref{Newton}, we add Step~4, that is inspired by \cite{chen1999proximal,Dan2002}, to check whether the prox-gradient residual
$\|\mathcal{G}(\hat{x}^k)\|$ of the inexact solution $\hat{x}^k$ to subproblem \eqref{subproblem} decreases to be under a given fixed ratio, which is smaller than one, times the previous value. If it
does, then we will update the next iterate $x^{k+1}$ by using the Newton direction $d^k$ with a unit step size to let $x^{k+1} = \hat{x}^k$ and skip the backtracking line search in Step~5. Otherwise, we
use the conventional backtracking line search procedure in Step~5 to find a conservative step size for updating the next iterate. It is shown in Theorem~\ref{Newton_thm1} below that Step~4
always gives us a unit step size when the iterate $x^k$ is close to the solution under the metric $q$-subregularity condition.\vspace*{0.03in}

In the rest of this section, we show that the proposed algorithm {\em globally converges} under the mild standing assumptions, which are imposed above and will not be repeated. Let us start with the
following lemma providing a subgradient estimate for subproblem \eqref{subproblem1} at the approximate solution.

\begin{lemma}\label{lem1} Given an approximate solution $\hat{x}^k$ to \eqref{subproblem1}, there exists a vector $e_k\in\R^n$ such that
\begin{equation}\label{ek}
e_k\in\nabla f(x^k)+H_k(\hat{x}^k-x^k)+\partial g(\hat{x}^k-e_k)\;\mbox{ and }\;\|e_k\|\le\nu\min\big\{\|\mathcal{G}(x^k)\|,\|\mathcal{G}(x^k)\|^{1+\varrho}\big\}.
\end{equation}
\end{lemma}
\revise{\begin{proof}
Let $e_k:=r_k(\hat{x}^k)=\hat{x}^k-\mathrm{Prox}_{g}(\hat{x}^k-\nabla f(x^k)-H_k(\hat{x}^k-x^k))$. Then we have
\[
e_k\in\nabla f(x^k)+H_k(\hat{x}^k-x^k)+\partial g(\hat{x}^k-e_k),
\]
which follow from \eqref{prox-map}.
Using finally the inexact conditions \eqref{inexact_condition_1} for $\hat x^k$, we verify the claim of the lemma.
\end{proof}}%

The next lemma provides elaborations on Step~5 of the proposed algorithm with the decreasing of the cost function in \eqref{prob_ori_1} by the backtracking line search.

\begin{lemma}\label{decrease} Let $t_k$ be chosen by the backtracking line search in Step~{\rm 5} of Algorithm~{\rm\ref{Newton}} at iteration $k$. Then we have the step size estimate
\begin{equation}\label{tk}
t_k \ge \frac{\gamma(1-\theta)\alpha_k}{L_1}
\end{equation}
with the cost function decrease satisfying
\begin{equation}\label{decrease_eq}
F(x^{k+1})-F(x^k)\le-\frac{\gamma\theta(1-\theta)}{2L_1}\left(\frac{(1-\nu)\alpha_k}{1+M+\alpha_k}\right)^2\|\mathcal{G}(x^k)\|^2.
\end{equation}
\end{lemma}
\begin{proof} Since $\hat{x}^k$ is an inexact solution to \eqref{subproblem} obeying the conditions in \eqref{inexact_condition_1}, it follows that
\[
0\ge q_k(\hat{x}^k)-q_k(x^k)=l_k(\hat{x}^k)-l_k(x^k)+\frac{1}{2}(\hat{x}^k-x^k)^TH_k(\hat{x}^k-x^k),
\]
where $l_k$ is the linear part of $q_k$ defined in \eqref{lin}. This yields
\begin{equation}\label{decrease_eq1}
l_k(x^k)-l_k(\hat{x}^k)\ge\frac{1}{2}(\hat{x}^k-x^k)^TH_k(\hat{x}^k-x^k)\ge\frac{1}{2}\alpha_k\|\hat{x}^k-x^k\|^2.
\end{equation}
By $\mathcal{G}(x^k)=x^k-\mathrm{Prox}_{g}(x^k-\nabla f(x^k))$ we deduce from the stationary and subdifferential sum rules that
\[
\mathcal{G}(x^k)\in\nabla f(x^k)+\partial g\big(x^k-\mathcal{G}(x^k)\big).
\]
Furthermore, Lemma~\ref{lem1} gives us the condition $e_k\in\nabla f(x^k)+H_k(\hat{x}^k-x^k)+\partial g(\hat{x}^k-e_k)$ for $\hat x^k$ with $e_k$ satisfying the estimate
$\|e_k\|\le\nu\|\mathcal{G}(x^k)\|$. The monotonicity of the subgradient mapping $\partial g$ ensures that
\[
\big\langle \mathcal{G}(x^k)+H_k(\hat{x}^k-x^k)-e_k,x^k-\mathcal{G}(x^k)-\hat{x}^k+e_k\big\rangle\ge 0,
\]
which therefore leads us to the inequality
\begin{equation*}
\begin{aligned}
\| \mathcal{G}(x^k) - e_k\|^2 & \le \|\mathcal{G}(x^k)\|^2 -2\big\langle e_k,\mathcal{G}(x^k)\big\rangle +\|e_k\|^2+(\hat{x}^k-x^k)^TH_k(\hat{x}^k-x^k) \\
&\le\big\langle \mathcal{G}(x^k) - e_k,x^k-\hat{x}^k+H_k(x^k-\hat{x}^k)\big\rangle \\
&\le\big\| \mathcal{G}(x^k) - e_k\big\| \cdot \big\| x^k-\hat{x}^k+H_k(x^k-\hat{x}^k)\big\|.
\end{aligned}
\end{equation*}
Using again the condition $\|e_k\|\le\nu\|\mathcal{G}(x^k)\|$ together with $\|B_k\|\le M$ from \eqref{assum_B_bounded} results in
\begin{equation*}
\|\mathcal{G}(x^k)\| \le \| \mathcal{G}(x^k) - e_k\| + \|e_k\| \le(1+M+\alpha_k)\|\hat{x}^k-x^k\|+ \nu\|\mathcal{G}(x^k)\|.
\end{equation*}
Remembering the choice of $\nu\in[0, 1)$, we estimate the prox-gradient mapping \eqref{prox-gra} at the iterate $x^k$ by
\begin{equation}\label{decrease_eq2}
\|\mathcal{G}(x^k)\|\le\frac{1+M+\alpha_k}{1-\nu}\|\hat{x}^k-x^k\|.
\end{equation}

Next let us show that the backtracking line search along the direction $d^k=\hat{x}^k- x^k$ in Step~5 is well-defined and the proposed step size ensures a sufficient decrease in the cost function $F$.
It follows from the Lipschitz continuity of $\nabla f$ that
\[
f(x^k+\tau d^k)\le f(x^k)+\tau\nabla f(x^k)^Td^k+\frac{L_1}{2}\tau^2\|d^k\|^2\;\mbox{ for any }\;\tau\ge 0,
\]
and thus we deduce from the definition of $l_k$ in \eqref{lin} that
\[
F(x^k)-F(x^k+\tau d^k)\ge l_k(x^k)-l_k(x^k+\tau d^k)-\frac{L_1}{2}\tau^2\|d^k\|^2.
\]
This implies by the convexity of $g$ that
\[
l_k(x^k)-l_k(x^k+\tau d^k)\ge\tau\big(l_k(x^k)-l_k(x^k+d^k)\big).
\]
Combining the latter with \eqref{decrease_eq1} and using the choice of $\theta\in(0,1)$ yield the relationships
\begin{equation}\label{decrease_eq3}
\begin{aligned}
&F(x^k)-F(x^k+\tau d^k)-\frac{\theta\alpha_k\tau}{2}\|d^k\|^2\\
\ge~&l_k(x^k)-l_k(x^k+\tau d^k)-\frac{L_1}{2}\tau^2\|d^k\|^2-\frac{\theta\alpha_k\tau}{2}\|d^k\|^2\\
\ge~&\tau\big(l_k(x^k)-l_k(x^k+d^k)\big)-\frac{L_1}{2}\tau^2\|d^k\|^2-\frac{\theta\alpha_k\tau}{2}\|d^k\|^2\\
\ge~&(1-\theta)\tau\frac{\alpha_k}{2}\|d^k\|^2-\frac{L_1}{2}\tau^2\|d^k\|^2\\
=~&\frac{\tau}{2}\|d^k\|^2\big((1-\theta)\alpha_k-L_1\tau\big).
\end{aligned}
\end{equation}
This tells us that the backtracking line search criterion \eqref{line_search} fulfills when $0<\tau\le\frac{(1-\theta)\alpha_k}{L_1}$, and thus the step size $t_k$ satisfies the claimed condition
\eqref{tk}. Substituting now $\tau:=t_k\ge\frac{\gamma(1-\theta)\alpha_k}{L_1}$ into \eqref{decrease_eq3} and employing the estimate of $\|\mathcal{G}(x^k)\|^2$ from \eqref{decrease_eq2}, we arrive at the
inequalities
\[
\begin{aligned}
F(x^k)-F(x^k+t_k d^k)&\ge\frac{\theta\alpha_kt_k}{2}\|d^k\|^2\\
&\ge\frac{\gamma\theta(1-\theta)\alpha_k^2}{2L_1}\left(\frac{1-\nu}{1+M+\alpha_k}\right)^2\|\mathcal{G}(x^k)\|^2,
\end{aligned}
\]
which verify the decreasing condition \eqref{decrease_eq} and thus completes the proof of the lemma.
\end{proof}

Now we are ready to prove the global convergence of Algorithm \ref{Newton}. Define the sets
\begin{equation}\label{K}
K:=\big\{0,1,\ldots\big\}\;\mbox{ and }\;K_0:=\{0\}\cup\big\{k+1\in K ~\big|~\text{Step~$5$ is not applied at iteration $k$}\big\}.
\end{equation}

\begin{theorem}\label{algorithm_convergence}
Let $\{x^k\}$ be the sequence of iterates generated by Algorithm~{\rm\ref{Newton}} with an arbitrarily chosen starting point $x^0\in\R^n$ under the standing assumptions made. Then we have the residual
condition
\begin{equation}\label{glob}
\liminf_{k\rightarrow\infty}\|\mathcal{G}(x^k)\|=0
\end{equation}
along the prox-gradient mapping \eqref{prox-gra}. Furthermore, the boundedness of $\{x^k\}$ yields the convergence to the optimal value $\lim\limits_{k\rightarrow\infty}F(x^{k})=F^*$ and ensures that any
limiting point of $\{x^k\}$ is a global minimizer in \eqref{prob_ori_1}.
\end{theorem}
\begin{proof}
First we consider the case where the set $K_0$ is infinite. We can reorganize $K_0$ in such a way that $0=k_0<k_1<k_2<\ldots$. It follows from Step~4 of Algorithm~\ref{Newton} that the estimate
\[
\|\mathcal{G}(x^{k_{\ell+1}})\|\le\sigma\|\mathcal{G}(x^{k_{\ell}})\|\;\mbox{ whenever }\;\ell=0,1,\ldots
\]
holds with the chosen number $\sigma\in(0,1)$ in the algorithm, and therefore we get
\revise{\begin{equation*}
0\le \liminf_{k\rightarrow\infty}\|\mathcal{G}(x^k)\| \le \limsup\limits_{\ell\rightarrow\infty}\|\mathcal{G}(x^{k_{\ell}})\|\le\lim\limits_{\ell\rightarrow\infty}\sigma^{\ell}\|\mathcal{G}(x^{k_{0}})\|=0,
\end{equation*}}%
which clearly yields \eqref{glob}. The continuity of $\mathcal{G}(\cdot)$ ensures that $\|\mathcal{G}(\bar{x})\|=0$ for a limiting point $\bar{x}$ of the sequence $\{x^k\}_{k\in K_0}$, and thus $\bar{x}
\in\mathcal{X}^*$.
Consider now any limiting point $\ox$ of the entire sequence of iterates $\{x^k\}_{k\in K}$. If there exists $\bar{k}$ such that $k\in K_0$ for all $k\ge\bar{k}$, it is easy to see that $\ox$ is a global
minimizer of \eqref{prob_ori_1}. Otherwise, for any $k\notin K_0$, denote by $k_{\ell}\in K_0$ the largest number satisfying $k_{\ell}<k$, and hence we get the following estimate from Step~5:
\[
 F^*\le F(x^{k})\le F(x^{k-1})\le\ldots\le F(x^{k_\ell}).
\]
When the sequence $\{x^k\}_{k\in K_0}$ is bounded, since any limiting point of $\{x^k\}_{k\in K_0}$ is a global minimizer of \eqref{prob_ori_1} as already shown, it follows that
$\lim\limits_{\ell\rightarrow\infty}F(x^{k_\ell})=F^*$. This readily verifies by the constructions above that $\lim\limits_{k\rightarrow\infty}F(x^{k})=F^*$, and thus any limiting point
of $\{x^k\}_{k\in K}$ provides a global minimum to \eqref{prob_ori_1}.

Next we consider the case where $K_0$ is finite and denote $\bar{k}:=\max_{k\in K_0}k$. It follows from Lemma~\ref{decrease} that for any $k>\bar{k}$ we get
\[
F(x^{k+1})-F(x^k)\le-\frac{\gamma\theta(1-\theta)}{2L_1}\left(\frac{(1-\nu)\alpha_k}{1+M+\alpha_k}\right)^2\|\mathcal{G}(x^k)\|^2,
\]
which therefore tells us that
\[
\sum_{k=\bar{k}}^{\infty}\frac{\gamma\theta(1-\theta)}{2L_1}\left(\frac{(1-\nu)\alpha_k}{1+M+\alpha_k}\right)^2\|\mathcal{G}(x^k)\|^2\le F(x^{\bar{k}})-F^*\le 0.
\]
The latter implies in turn that
\[
\lim_{k\rightarrow \infty}\frac{\gamma\theta(1-\theta)}{2L_1}\left(\frac{(1-\nu)\alpha_k}{1+M+\alpha_k}\right)^2\|\mathcal{G}(x^k)\|^2=0.
\]
Remembering the choice of \revise{$\alpha_k=\min \left\{ \bar{\alpha}, c\|\mathcal{G}(x^k)\|^{\rho}\right\}$} with $\bar{\alpha}, c, \rho>0$ ensures that
\[
\lim_{k\rightarrow\infty}\|\mathcal{G}(x^k)\|=0,
\]
hence \eqref{glob} holds. This readily verifies by the continuity of $\mathcal{G}(\cdot)$ that any limiting point of $\{x^k\}_{k\in K}$ provides a global minimum to \eqref{prob_ori_1}.
\end{proof}

We conclude this section with a consequence of Theorem~\ref{algorithm_convergence} giving an easily verifiable condition for the boundedness of the sequence of iterates in Algorithm~\ref{Newton}. Recall
that a function $\ph\colon\R^n\to\oR$ is {\em coercive} if $\ph(x)\to\infty$ provided that $\|x\|\to\infty$.

\begin{corollary}\label{coer} In addition to the standing assumptions imposed above, suppose that the cost function $F$ in \eqref{prob_ori_1} is coercive. Then we have $\lim\limits_{k\rightarrow
\infty}F(x^{k})=F^*$ for the sequence of iterates $\{x^k\}$ generated by Algorithm~{\rm\ref{Newton}}, and any limiting point of $\{x^k\}$ is a global minimizer of \eqref{prob_ori_1}.
\end{corollary}
\begin{proof} According to Steps~4 and 5 of Algorithm~\ref{Newton}, the sequence $\{x^k\}$ generated by the algorithm satisfies the condition $F(x^k)\le C$ for all $k$. Then the coercivity of $F$
implies that the sequence $\{x^k\}$ is bounded. Thus we deduce the conclusions of the corollary from Theorem~\ref{algorithm_convergence}.
\end{proof}\vspace*{-0.25in}

\section{Fast Local Convergence under Metric $q$-Subregularity}\label{sec:5}\vspace*{-0.05in}

This section is devoted to the local convergence of the proximal Newton-type Algorithm~\ref{Newton} under the imposed metric $q$-subregularity in both cases where $q\in(0,1]$ and
$q>1$. In the first case, which is referred to as the H\"older metric subregularity, we do not consider any $q\in(0,1]$, but precisely specify the lower bound of $q$ and respectively modify some
parameters of our algorithm. Namely, for the case where $q = 1$, i.e., the metric subregularity assumption of the subgradient mapping in \eqref{prob_ori_1} holds, the main result here establishes
superlinear local convergence rates depending on the selected exponent $\rho\in(0,1]$ in the algorithm, which gives us the {\em quadratic} convergence in the case where $\rho=1$. Our analysis partly
follows the scheme of \cite{fischer2002local} for a Newtonian algorithm to solve generalized equations with nonisolated solutions under certain Lipschitzian properties of perturbed solution sets. Note that the imposed metric subregularity allows us to avoid limitations of the line search procedure (needed for establishing the global convergence of our algorithm in Section~\ref{sec:2} that is not addressed in \cite{fischer2002local}) to achieve now the fast local convergence.
The imposed metric $q$-subregularity assumption is weaker for $q<1$ than the metric subregularity, but allows us to achieve a local {\em superlinear} (while not quadratic) convergence of the algorithm.
In the other case where $q>1$, we achieve a {\em higher-than-quadratic} rate of the local convergence of the proposed algorithm.\vspace*{0.03in}

Starting with the H\"older metric subregularity, we first provide the following norm estimate of directions $d^k$ in the proposed Algorithm~\ref{Newton}.

\begin{lemma}\label{holder_d_k} Let $\{x^k\}$ be the sequence
generated by Algorithm~{\rm\ref{Newton}}, and let
$\bar{x}\in\mathcal{X}^*$ be any limiting point of the sequence
$\{x^k\}$. In addition to the standing assumptions, suppose that the
subgradient mapping $\nabla f(x)+\partial g(x)$ is metrically
$q$-subregular at $(\bar{x},0)$ for some $q\in(0,1]$, that the
Hessian $\nabla^2f$ is locally Lipschitzian around $\bar{x}$, that
\revise{$\alpha_k=\min \left\{ \bar{\alpha}, c\|\mathcal{G}(x^k)\|^{\rho}\right\}$} with $\bar{\alpha}, c > 0$,
$\rho\in(0,q]$, and $\varrho\ge \rho$ in
Algorithm~{\rm\ref{Newton}}, and that the estimate $\|B_k-\nabla^2
f(x^k)\|\le C_1\,\mathrm{dist}(x^k;\mathcal{X}^*)$ holds with some
constant $C_1>0$. Then there exist positive numbers $\epsilon$ and
$c_1$ such that for $d^k:=\hat{x}^k-x^k$ we have the direction
estimate \begin{equation}\label{dk1} \|d^k\|\le
c_1\,\mathrm{dist}(x^k;\mathcal{X}^*)\;\mbox{ as
}\;x^k\in\mathbb{B}_{\epsilon}(\bar{x}). \end{equation}
\end{lemma}

\begin{proof} Remembering that $\hat{x}^k$ is an inexact solution to
\eqref{subproblem} satisfying conditions
\eqref{inexact_condition_1}, we apply Lemma~\ref{lem1} and find a
vector $e^k$ such that the relationships in \eqref{ek} hold.
Denoting by $\pi_x^k$ the (unique) projection of $x^k$ onto the
solution map $\mathcal{X}^*$, we get by basic convex analysis that
$0\in\nabla f(\pi_x^k)+\partial g(\pi_x^k)$ and thus
\begin{equation*} \nabla f(x^k)-\nabla
f(\pi_x^k)+H_k(\pi_x^k-x^k)\in\nabla
f(x^k)+H_k(\pi_x^k-x^k)+\partial g(\pi_x^k). \end{equation*} Since
the mapping $\nabla f(x^k)+H_k(\cdot-x^k)+\partial g(\cdot)$ is
strongly monotone on $\R^n$ with constant $\alpha_k$, we have
\begin{equation*}\label{mon} \langle\nabla f(x^k)-\nabla
f(\pi_x^k)+H_k(\pi_x^k-x^k)-e_k+H_ke_k,\pi_x^k-\hat{x}^k+e_k\rangle\ge\alpha_k\|\pi_x^k-\hat{x}^k+e_k\|^2.
\end{equation*} Combining the above with the algorithm constructions
gives us the estimates \begin{equation}\label{mon2}
\begin{aligned}
\|\pi_x^k-\hat{x}^k+e_k\|\le\frac{1}{\alpha_k}\|\nabla f(x^k)-\nabla f(\pi_x^k)+H_k(\pi_x^k-x^k)-e_k+H_ke_k\|\\
\le\frac{1}{\alpha_k}\left(\|\nabla f(x^k)+\nabla^2 f(x^k)(\pi_x^k-x^k)-\nabla f(\pi_x^k)\|+\|(H_k-\nabla^2 f(x^k))(\pi_x^k-x^k)\|+\|e_k-H_ke_k\|\right)\\
\le\frac{1}{\alpha_k}\Big(\|\nabla f(x^k)+\nabla^2 f(x^k)(\pi_x^k-x^k)-\nabla f(\pi_x^k)\|+\|B_k-\nabla^2 f(x^k)\|\cdot\|x^k-\pi_x^k\|\\
\quad+\alpha_k\|x^k-\pi_x^k\|+(1+M)\|e_k\|\Big)\\
\le\frac{1}{\alpha_k}\Big(\|\nabla f(x^k)+\nabla^2 f(x^k)(\pi_x^k-x^k)-\nabla f(\pi_x^k)\|+\|B_k-\nabla^2 f(x^k)\|\,\mathrm{dist}(x^k;\mathcal{X}^*)\\
\quad+\alpha_k\mathrm{dist}(x^k;\mathcal{X}^*)+(1+M)\nu\|\mathcal{G}(x^k)\|^{1+\varrho}\Big),
\end{aligned} \end{equation} where the third inequality follows from
the choice of $H_k=B_k+\alpha_kI$ while the fourth inequality is
implied by $\|e_k\|\le\nu\|\mathcal{G}(x^k)\|^{1+\varrho}$. Since
the Hessian mapping $\nabla^2f$ is locally Lipschitzian around
$\bar{x}$, there exist positive numbers $\epsilon_2$ and $L_2$ such
that for any $x,y\in\mathbb{B}_{\epsilon_2}(\bar{x})$ we get \[
\|\nabla f(x)+\nabla^2 f(x)(y-x)-\nabla
f(y)\|\le\frac{L_2}{2}\|x-y\|^2.
\]
Furthermore, the imposed
assumption that $\|B_k-\nabla^2 f(x^k)\|\le
C_1\,\mathrm{dist}(x^k;\mathcal{X}^*)$ and the fact that
$x^k\in\mathbb{B}_{\epsilon_2}(\bar{x})$ implying
$\pi_x^k\in\mathbb{B}_{\epsilon_2}(\bar{x})$ give us the estimate
\begin{equation*} \begin{aligned}
\|\pi_x^k-\hat{x}^k+e_k\|&\le\frac{1}{\alpha_k}\Big((\frac{L_2}{2}+C_1)\mathrm{dist}(x^k;\mathcal{X}^*)^2
+
\alpha_k\mathrm{dist}(x^k;\mathcal{X}^*)+(1+M)\nu\|\mathcal{G}(x^k)\|^{1+\varrho}\Big)
\end{aligned} \end{equation*} provided
$x^k\in\mathbb{B}_{\epsilon_2}(\bar{x})$. Next we employ the
relationships \begin{equation}\label{dist-est2}
\|d^k\|=\|\hat{x}^k-x^k\|\le\|\pi_x^k-\hat{x}^k+e_k\|+\|\pi_x^k-x^k\|+\|e_k\|\;\mbox{
with }\;\|e_k\|\le\nu\|\mathcal{G}(x^k)\|^{1+\varrho} \end{equation}
together with
$\|\mathcal{G}(x^k)\|\le(2+L_1)\,\mathrm{dist}(x^k;\mathcal{X}^*)$
by Proposition~\ref{bound_rxk} to obtain that
\begin{equation}\label{dist-est} \begin{aligned}
\alpha_k\|d^k\|&\le\left(\frac{L_2}{2}+C_1\right)\mathrm{dist}(x^k;\mathcal{X}^*)^2+2\alpha_k\,\mathrm{dist}(x^k;\mathcal{X}^*)\\
&\quad+(1+M+\alpha_k)\nu\|\mathcal{G}(x^k)\|^{\rho}(2+L_1)^{1+\varrho-\rho}
\mathrm{dist}(x^k;\mathcal{X}^*)^{1+\varrho-\rho}\end{aligned}
\end{equation} provided that
$x^k\in\mathbb{B}_{\epsilon_2}(\bar{x})$. The assumed metric
$q$-subregularity of $\nabla f(x)+\partial g$ gives us by
Proposition~\ref{lem3} numbers $\epsilon_1,\kappa_1>0$ with
\begin{equation*}
\mathrm{dist}(x;\mathcal{X}^*)\le\kappa_1\|\mathcal{G}(x)\|^{q}\;\mbox{
for all }\;x\in\mathbb{B}_{\epsilon_1}(\bar{x}). \end{equation*}
Supposing without loss of generality that $\epsilon_1\le
\min\{1,\epsilon_2\}$ \revise{and $\alpha_k=\min \left\{ \bar{\alpha}, c\|\mathcal{G}(x^k)\|^{\rho}\right\} = c\|\mathcal{G}(x^k)\|^{\rho}$ when $x^k\in\mathbb{B}_{\epsilon_1}(\bar{x})$ and remembering that
$\rho\in (0,q]$ imply that }%
\begin{equation}\label{ak}
\alpha_k=c\|\mathcal{G}(x^k)\|^{\rho}\ge
c\kappa_1^{-\frac{\rho}{q}}\mathrm{dist}(x;\mathcal{X}^*)^{\frac{\rho}{q}}
\ge c\kappa_1^{-\frac{\rho}{q}}\mathrm{dist}(x;\mathcal{X}^*)
\;\mbox{ as }\;x^k\in\mathbb{B}_{\epsilon_1}(\bar{x}).
\end{equation} Since $\rho\in (0,q]$ and $\varrho\ge \rho$, we
deduce from \eqref{dist-est} and \eqref{ak} the existence of
positive numbers $\epsilon$ and $c_1$ ensuring the fulfillment of
estimate \eqref{dk1} claimed in the lemma. \end{proof}

Next, under the imposed H\"older metric subregularity, we show that
the set $K_0$ defined in \eqref{K} is infinite.

\begin{lemma}\label{K0} Let $\{x^k\}$ be the sequence generated by
Algorithm~{\rm\ref{Newton}}, and let $\bar{x}\in\mathcal{X}^*$ be
any limiting point of $\{x^k\}$. In addition to the
standing assumptions, suppose that the mapping $\nabla
f(x)+\partial g(x)$ is metrically $q$-subregular at $(\bar{x},0)$
for some $q\in(0,1]$, that the Hessian $\nabla^2 f$ is locally
Lipschitzian around $\bar{x}$, that
\revise{$\alpha_k=\min \left\{ \bar{\alpha}, c\|\mathcal{G}(x^k)\|^{\rho}\right\}$} with $\bar{\alpha}, c > 0$,
$\rho\in(0,q]$, and $\varrho\ge \rho$ in
Algorithm~{\rm\ref{Newton}}, and that the estimate $\|B_k-\nabla^2
f(x^k)\|\le C_1\,\mathrm{dist}(x^k;\mathcal{X}^*)$ holds with some
constant $C_1>0$. Then the set $K_0$ defined in \eqref{K} is
infinite. \end{lemma} \begin{proof} On the contrary, suppose that
$K_0$ is finite and denote $\bar{k}:=\max_{k\in K_0}k$. Arguing as
in the proof of Lemma~\ref{algorithm_convergence} tells us that
\[\lim_{k\rightarrow\infty}\|\mathcal{G}(x^k)\|=0.\] According to
Lemma~\ref{holder_d_k}, there exist positive numbers $\epsilon$ and
$c_1$ such that for $d^k:=\hat{x}^k-x^k$ we have \begin{equation*}
\|d^k\|\le c_1\,\mathrm{dist}(x^k;\mathcal{X}^*)\;\mbox{ as
}\;x^k\in\mathbb{B}_{\epsilon}(\bar{x}), \end{equation*} which
implies that $\liminf_{k\rightarrow\infty} \|d^k\| = 0$. Combining
this with the $(2+L_1)$-Lipschitz continuity of $\|\mathcal{G}(x)\|$
as follows from Proposition \ref{bound_rxk0} gives us
\begin{equation*}
\liminf_{k\rightarrow\infty}\|\mathcal{G}(\hat{x}^k)\| \le
\lim_{k\rightarrow\infty} \|\mathcal{G}(x^k)\| +
\liminf_{k\rightarrow\infty} (2+L_1)\|d^k\| = 0. \end{equation*}
Hence there exists $k>\bar{k}$ such that $k\in K_0$; a contradiction
showing that the set $K_0$ is infinite. \end{proof}

Having Lemmas~\ref{holder_d_k}, \ref{K0}, and the previous estimates
in hand, next we derive the following fast local convergence result
for Algorithm~\ref{Newton} with a particular choice of parameters
under the imposed H\"older metric subregularity of the subgradient
mapping $\nabla f+\partial g$ with an appropriate factor $q$.

\begin{theorem}\label{Newton_thm1} Let $\{x^k\}$ be the sequence
generated by Algorithm~{\rm\ref{Newton}}, and let
$\bar{x}\in\mathcal{X}^*$ be any limiting point of the sequence
$\{x^k\}_{k\in K_0}$, where $K_0$ is defined in \eqref{K}. In
addition to the standing assumptions, suppose that the subgradient
mapping $\nabla f(x)+\partial g(x)$ is metrically $q$-subregular at
$(\bar{x},0)$ with $q\in(\frac{1}{2},1]$, that the Hessian mapping
$\nabla^2 f$ is locally Lipschitzian around $\bar{x}$, that
\revise{$\alpha_k=\min \left\{ \bar{\alpha}, c\|\mathcal{G}(x^k)\|^{\rho}\right\}$}, $\bar{\alpha}, c > 0$, $\rho \in
[\frac{1}{2},q]$, and $\varrho\ge \rho$ in
Algorithm~{\rm\ref{Newton}}, and that $\|B_k-\nabla^2
f(x^k)\|=O(\|\mathcal{G}(x^k)\|)$. Then there exists a natural
number $k_0$ such that \[ t_k=1\;\mbox{ for all }\;k\ge k_0, \] and
the sequence $\{x^k\}$ converges to the point $\bar{x}$.
Furthermore, this convergence is superlinear with the rate of
$\rho+q > 1$ in the sense that there exist a positive number $C_0$
and a natural number $k_0$ for which \begin{equation}\label{super}
\|\mathcal{G}(x^{k+1})\| \le C_0 \|\mathcal{G}(x^k)\|^{\rho+q}
\;\mbox{ whenever }\;k\ge k_0, \end{equation} and
$\mathrm{dist}(x^k;\mathcal{X}^*)$ converges R-superlinearly to $0$
in the sense that $\lim_{k\rightarrow \infty}
\sqrt[k]{\mathrm{dist}(x^k;\mathcal{X}^*)} = 0$.

In particular, when the subgradient mapping $\nabla f(x)+\partial
g(x)$ is metrically subregular at $(\bar{x},0)$, i.e., $q = 1$, and
when $\rho=1$, we have the quadratic convergence of $x^k\to\ox$ with
the exponent $\rho + q=2$ in \eqref{super}, and there exists a
positive constant $\tilde{C}_0$ such that
\begin{equation}\label{quad} \mathrm{dist}(x^{k+1};\mathcal{X}^*)
\le \tilde{C}_0 \mathrm{dist}(x^k;\mathcal{X}^*)^2 \;\mbox{ whenever
}\;k\ge k_0. \end{equation} \end{theorem}

\begin{proof} Observe first that the assumed metric
$q$-subregularity of the mapping $\nabla f(x)+\partial g(x)$ at
$(\bar{x},0)$ gives us positive numbers $\epsilon_1$ and $\kappa_1$
such that for all $p$ near $0\in\R^n$ we have the inclusion
\begin{equation}\label{H_error_bound_eq}
\Sigma(p)\cap\mathbb{B}_{\epsilon_1}(\bar{x})\subset\mathcal{X}^*+\kappa_1\|p\|^{q}\mathbb{B}
\;\mbox{ with }\;\Sigma(p):=\big\{x\in\mathbb{R}^n~\big|~p\in\nabla
f(x)+\partial g(x)\big\}. \end{equation} Employing
Proposition~\ref{lem3} allows us to find $\kappa_2>1$ ensuring the
estimate \begin{equation}\label{thm3_eq2}
\mathrm{dist}(x;\mathcal{X}^*)\le\kappa_2\|\mathcal{G}(x)\|^{q}\;\mbox{
whenever }\;x\in\mathbb{B}_{\epsilon_1}(\bar{x}). \end{equation}
Since $\|B_k-\nabla^2f(x^k)\|=O(\|\mathcal{G}(x^k)\|)$, we deduce
from Proposition~\ref{bound_rxk} the existence of $C_1>0$ with
\begin{equation}\label{bk} \|B_k-\nabla^2 f(x^k)\|\le
C_1\,\mathrm{dist}(x^k;\mathcal{X}^*). \end{equation} Recalling that
$\hat{x}^k$ is an inexact solution of \eqref{subproblem} satisfying
\eqref{inexact_condition_1} and using Lemma~\ref{lem1} give us
\begin{equation*} e_k\in\nabla f(x^k)+H_k(\hat{x}^k-x^k)+\partial
g(\hat{x}^k-e_k)\;\mbox{ with
}\;\|e_k\|\le\nu\min\big\{\|\mathcal{G}(x^k)\|,\|\mathcal{G}(x^k)\|^{1+\varrho}\big\}.
\end{equation*} By setting $\tilde{x}^k:=\hat{x}^k-e_k$, we have the
inclusion \[e_k-H_ke_k\in\nabla f(x^k)+H_k(\tilde{x}^k-x^k)+\partial
g(\tilde{x}^k), \] which implies therefore that
\begin{equation}\label{Rk} \mathcal{R}_k(\tilde{x}^k, x^k):=\nabla
f(\tilde{x}^k)-\nabla f(x^k)-H_k(\tilde{x}^k-x^k)+e_k-H_k
e_k\in\nabla f(\tilde{x}^k)+\partial g(\tilde{x}^k). \end{equation}
The latter reads, by the above definition of the perturbed solution
map $\Sigma(p)$, that
$\tilde{x}^k\in\Sigma(\mathcal{R}_k(\tilde{x}^k,x^k))$. Since
$\nabla^2 f$ is locally Lipschitzian around $\bar{x}$, there exist
numbers $L_2,\epsilon_2>0$ such that
\begin{equation}\label{newton:theorem2_lip} \|\nabla f(x)+\nabla^2
f(x)(y-x)-\nabla f(y)\|\le\frac{L_2}{2}\|x-y\|^2\;\mbox{ for any
}\;x,y\in\mathbb{B}_{\epsilon_2}(\bar{x}). \end{equation} Then
choose by Lemma~\ref{holder_d_k} a small number $0 <\ve_1
<\min\{1,\epsilon_2\}$ such that
\begin{equation}\label{Newton_thm2_dk} \|d^k\|\le c_1
\mathrm{dist}(x^k;\mathcal{X}^*) \;\mbox{ for all
}\;x^k\in\mathbb{B}_{\epsilon_1}(\bar{x}) \end{equation} with some
$c_1>0$. Since
$\|\tilde{x}^k-\bar{x}\|\le\|x^k-\bar{x}\|+\|d^k\|+\|e_k\|$ with
$\|d^k\|\rightarrow 0$ and $\|e_k\|\rightarrow 0$ when
$x^k\rightarrow\bar{x}$ as $k\to\infty$, we find
$0<\epsilon_3\le\epsilon_1$ such that
$\tilde{x}^k\in\mathbb{B}_{\epsilon_1}(\bar{x})$ whenever
$x^k\in\mathbb{B}_{\epsilon_3}(\bar{x})$. We also assume that
$\epsilon_3$ is sufficiently small with $\|\mathcal{G}(x)\| < 1$ for
all $x\in\mathbb{B}_{\epsilon_3}(\bar{x})$. This leads us to the
relationships \begin{equation}\label{Newton_thm_2_eq1}
\begin{aligned}
\|\mathcal{R}_k(\tilde{x}^k,x^k)\|&=\|\nabla f(\tilde{x}^k)-\nabla f(x^k)-H_k(\tilde{x}^k-x^k)+e_k-H_ke_k\|\\
&=\|\nabla f(\tilde{x}^k)-\nabla f(x^k)-(B_k+\alpha_kI)(\tilde{x}^k-x^k)+e_k-H_ke_k\|\\
&\le\|\nabla f(\tilde{x}^k)-\nabla f(x^k)-\nabla^2 f(x^k)(\tilde{x}^k-x^k)\|+\|B_k-\nabla^2 f(x^k)\|\cdot\|\tilde{x}^k-x^k\|+\alpha_k\|\tilde{x}^k-x^k\|\\
&\quad+(1+M)\|e_k\|\\
&\le\frac{L_2}{2}\|\tilde{x}^k-x^k\|^2+C_1\,\mathrm{dist}(x^k;\mathcal{X}^*)\|\tilde{x}^k-x^k\|
+\alpha_k\|\tilde{x}^k-x^k\|+(1+M)\nu\|\mathcal{G}(x^k)\|^{1+\varrho}
\end{aligned} \end{equation} if
$x^k\in\mathbb{B}_{\epsilon_3}(\bar{x})$, where the second
inequality follows from \eqref{newton:theorem2_lip}, $\|B_k-\nabla^2
f(x^k)\|\le C_1\,\mathrm{dist}(x^k;\mathcal{X}^*)$, and
$\|e_k\|\le\nu\|\mathcal{G}(x^k)\|^{1+\varrho}$. We have by
Proposition~\ref{bound_rxk} that
$\|e_k\|\le\nu\|\mathcal{G}(x^k)\|\le\nu(2+L_1)\mathrm{dist}(x^k;\mathcal{X}^*)$
for this choice of $x^k$. Thus it follows that \[
\|\tilde{x}^k-x^k\|\le\|\hat{x}^k-x^k\|+\|e^k\|=\|d^k\|+\|e_k\|\le\|d^k\|+\nu(2+L_1)\mathrm{dist}(x^k;\mathcal{X}^*),
\] which being combined with \eqref{Newton_thm_2_eq1},
Proposition~\ref{bound_rxk}, and $\varrho\ge \rho$ gives us $c_2>0$
such that \[ \|\mathcal{R}_k(\tilde{x}^k,x^k)\| \le
c_2\|d^k\|^2+c_2\,\mathrm{dist}(x^k;\mathcal{X}^*)^2+c_2\alpha_k\mathrm{dist}(x^k;\mathcal{X}^*)+\alpha_k\|d^k\|
\;\mbox{ for all }\;x^k\in\mathbb{B}_{\epsilon_3}(\bar{x}).\] Then
the direction estimate \eqref{Newton_thm2_dk} together with
\eqref{thm3_eq2} and the one of
\revise{$\alpha_k=\min \left\{ \bar{\alpha}, c\|\mathcal{G}(x^k)\|^{\rho}\right\}$} ensures the existence of a
positive constant $c_3$ ensuring the estimates
\begin{equation}\label{assump_Rk} \begin{aligned}
\|\mathcal{R}_k(\tilde{x}^k,x^k)\|  & \le  c_2 (c_1^2 + 1) \mathrm{dist}(x^k;\mathcal{X}^*)^2 +\alpha_k\left(c_1 + c_2\right)\mathrm{dist}(x^k;\mathcal{X}^*) \\
& \le  c_2 (c_1^2 + 1) \kappa_2^2\|\mathcal{G}(x^k)\|^{2q} +\alpha_k\left(c_1 + c_2 \right)\kappa_2\|\mathcal{G}(x^k)\|^{q} \\
& \le  c_3 \max\big\{ \|\mathcal{G}(x^k)\|^{2q}, \|\mathcal{G}(x^k)\|^{\rho+q}\big\} \\
&\le  c_3 \|\mathcal{G}(x^k)\|^{\rho+q} \;\mbox{ for all
}\;x^k\in\mathbb{B}_{\epsilon_3}(\bar{x}), \end{aligned}
\end{equation} where the last inequality follows from
$\|\mathcal{G}(x^k)\| \le 1$ and $\rho \le q$. Recalling that
$\mathcal{R}_k(\tilde{x}^k,x^k) \in\nabla f(x)+\partial g(x)$ and
using $\|\mathcal{G}(x)\|\le\mathrm{dist}(0;\nabla f(x)+\partial
g(x))$, which comes from Proposition \ref{bound_rxk0}, yield \[
\|\mathcal{G}(\tilde{x}^k)\| \le \|\mathcal{R}_k(\tilde{x}^k,x^k)\|
\le c_3 \|\mathcal{G}(x^k)\|^{\rho+q}\;\mbox{ for all
}\;x^k\in\mathbb{B}_{\epsilon_3}(\bar{x}). \] Combining the latter
with the $(2+L_1)$-Lipschitz continuity of $\|\mathcal{G}(x)\|$,
which comes from Proposition \ref{bound_rxk0}, and with
$\|e_k\|\le\nu\|\mathcal{G}(x^k)\|^{1+\varrho}$ as $\varrho\ge
\rho$, gives us \begin{equation}\label{thm1:sup}
\|\mathcal{G}(\hat{x}^k)\| \le (c_3 + (2+L_1)\nu)
\|\mathcal{G}(x^k)\|^{\rho+q}, \end{equation} and thus we arrive at
the estimate \[ \|\mathcal{G}(\hat{x}^k)\| \le (c_3 + (2+L_1)\nu)
\|\mathcal{G}(x^k)\|^{\rho+q-1}\cdot \|\mathcal{G}(x^k)\|\] provided
$x^k\in\mathbb{B}_{\epsilon_3}(\bar{x})$. Since $\rho+q-1>0$ and
$\|\mathcal{G}(x^k)\| \le(2+L_1)\,\mathrm{dist}(x^k;\mathcal{X}^*)$,
which comes from Proposition~\ref{bound_rxk}, this allows us to find
$0<\epsilon_0<\epsilon_3$ such that \begin{equation}\label{thm_eq3}
\|\mathcal{G}(\hat{x}^k)\|  \le\sigma\, \|\mathcal{G}(x^k)\|
\;\mbox{ for }\;x^k\in\mathbb{B}_{\epsilon_0}(\bar{x}).
\end{equation} Remembering that $C>F(x^0)\ge F_*$ and that $F$ is
continuous on the open domain $\mathrm{dom}\,F$, we select a
positive number $\epsilon_0$ to be so small that
\begin{equation}\label{thm_eq2}
\sup_{x\in\mathbb{B}_{\epsilon_0}(\bar{x})}F(x)\le C. \end{equation}
Let us next we introduce the positive constants \begin{equation*}
\tilde{\sigma}:=\sigma^{q}<1\;\mbox{ and
}\;\tilde{\epsilon}:=\disp\min\left\{\left(\frac{1-\tilde{\sigma}}{2c_1\kappa_2(2+L_1)^q}\epsilon_0\right)^{\frac{1}{q}},
\;\frac{1}{2}\epsilon_0, \;\frac{1}{1+c_1}\epsilon_0 \right\}
\end{equation*} and show that if
$x^{k_0}\in\mathbb{B}_{\tilde{\epsilon}}(\bar{x})$ with some $k_0\in
K_0$, then for any $k\ge k_0$ we have \begin{equation}\label{k0}
k+1\in K_0,\;t_k=1,\;x^{k+1}=\hat{x}^{k},\;\mbox{ and
}\;x^{k+1}\in\mathbb{B}_{\epsilon_0}(\bar{x}). \end{equation} To
verify \eqref{k0}, set first $k:=k_0$ and deduce from
$x^k\in\mathbb{B}_{\tilde{\epsilon}}(\bar{x})$ that
\begin{equation*}
\|\hat{x}^k-\bar{x}\|\le\|x^k-\bar{x}\|+\|d_k\|\le\|x^k-\bar{x}\|+c_1\,\mathrm{dist}(x^k;\mathcal{X}^*)\le(1+c_1)\|x^k-\bar{x}\|\le\epsilon_0,
\end{equation*} where the second inequality comes from
\eqref{Newton_thm2_dk}. It follows from \eqref{thm_eq3} and $k_0\in
K_0$ that \begin{equation*} \|\mathcal{G}(\hat{x}^k)\|  \le\sigma\,
\|\mathcal{G}(x^k)\|=\sigma\vartheta_{k}. \end{equation*} Observe
also that \eqref{thm_eq2} obviously yields $F(\hat{x}^{k})\le C$.
Then by Step~4 of Algorithm \ref{Newton} we get $k+1\in K_0$,
$t_k=1$, $x^{k+1}=\hat{x}^{k}$,
$\vartheta_{k+1}=\|\mathcal{G}(x^{k+1})\|$, and
$x^{k+1}\in\mathbb{B}_{\epsilon_0}(\bar{x})$. To justify further
\eqref{k0} for any $k>k_0$, proceed by induction and suppose that
for all $k-1\ge\ell\ge k_0$ we have \begin{equation*} \ell+1\in
K_0,\;t_\ell=1,\;x^{\ell+1}=\hat{x}^{\ell},\;x^{\ell+1}\in\mathbb{B}_{\epsilon_0}(\bar{x}),\;\mbox{and
hence }\;\;\vartheta_{\ell}= \|\mathcal{G}(x^{\ell})\|,\;
\|\mathcal{G}(x^{\ell+1})\| \le \sigma\,\|\mathcal{G}(x^{\ell})\|.
\end{equation*} This readily implies the estimates
\begin{equation}\label{thm_eq4} \begin{aligned}
\|\hat{x}^{k}-x^{k_0}\| &\le\sum_{\ell=k_0}^{k}\|d^{\ell}\| \le\sum_{\ell=k_0}^{k}c_1\kappa_2 \,\| \mathcal{G}(x^{\ell})\|^q \le\sum_{\ell=k_0}^{k}c_1\kappa_2\tilde{\sigma}^{\ell-k_0}\,\| \mathcal{G}(x^{k_0})\|^q \\
&\le\frac{c_1\kappa_2}{1-\tilde{\sigma}}\,
\|\mathcal{G}(x^{k_0})\|^q \le
\frac{c_1\kappa_2(2+L_1)^q}{1-\tilde{\sigma}}\|x^{k_0}-\bar{x}\|^q,
\end{aligned} \end{equation} where the second inequality follows
from \eqref{thm3_eq2} and \eqref{Newton_thm2_dk}, while the last
inequality is a consequence of Proposition~\ref{bound_rxk}.
Therefore, it gives us the conditions \[
\|\hat{x}^k-\bar{x}\|\le\|\hat{x}^k-x^{k_0}\|+\|x^{k_0}-\bar{x}\|\le\frac{c_1\kappa_2(2+L_1)^q}{1-\tilde{\sigma}}\tilde{\epsilon}^q
+ \tilde{\epsilon}\le\epsilon_0. \] Arguing as above, we get that
$\|\mathcal{G}(\hat{x}^k)\|  \le \sigma\vartheta_{k}$ and
$F(\hat{x}^{k})\le C$, which ensures that \eqref{k0} holds for any
$k\ge k_0$ and thus verifies these conditions in the general case.

Now we prove the claimed convergence $x^k\to\ox$ as $k\to\infty$
with the convergence rate \eqref{super}, where $\ox$ is the
designated limiting point $\ox$ of the sequence $\{x^k\}_{k\in
K_0}$. As shown above, for any $k\ge k_0$ we have $k+1 \in K_0$,
$t_k=1$, $x^{k+1}=\hat{x}^{k}$, and
$x^{k+1}\in\mathbb{B}_{\epsilon_0}(\bar{x})$. Using the conditions
in \eqref{k0} and the arguments similarly to to the proof of
\eqref{thm_eq4}, we are able to show that
\begin{equation}\label{thm1:seq_bounded} \|x^k-\bar{x}\|\le
\frac{c_1\kappa_2(2+L_1)^q}{1-\tilde{\sigma}}\|x^{\tilde{k}} -
\bar{x}\|^q +\| x^{\tilde{k}}-\bar{x}\|\quad \text{for any}\quad
k\ge\tilde{k}, \end{equation} whenever $\tilde{k}\ge k_0$. This
shows that the sequence $\{x^k\}$ is bounded. Picking any limiting
point $\tilde{x}$ of $\{x^k\}$ and passing to the limit as
$k\rightarrow\infty $ in \eqref{thm1:seq_bounded} lead us to the
estimate \begin{equation*}
\|\tilde{x}-\bar{x}\|\le\frac{c_1}{1-\tilde{\sigma}}\|x^{\tilde{k}}-\bar{x}\|+\|x^{\tilde{k}}-\bar{x}\|.
\end{equation*} Recalling that $\bar{x}$ is a limiting point of
$\{x^k\}_{k\in K_0}$, we pass to the limit as $\tilde{k} \rightarrow
\infty$ in the estimate above and get $\|\tilde{x}-\bar{x}\|=0$,
which implies that $\{x^k\}$ converges to $\bar{x}$. Finally,
employing \eqref{thm1:sup} gives us numbers $C_0,k_0>0$ such that
the claimed condition \eqref{super} holds. Then
$\|\mathcal{G}(x^{k+1})\|^q \le C_0^q
(\|\mathcal{G}(x^k)\|^q)^{\rho+q}$ for any $k\ge k_0$, which means
that $\|\mathcal{G}(x^k)\|^q$ converges $Q$-superlinearly to $0$.
Combined the latter with \eqref{thm3_eq2} implies that
$\mathrm{dist}(x^k;\mathcal{X}^*)$ converges $R$-superlinearly to
$0$. When the subgradient mapping $\nabla f(x)+\partial g(x)$ is
metrically subregular at $(\bar{x},0)$, we have the following
condition by setting $q = \rho = 1$ in \eqref{super}:
\begin{equation*} \|\mathcal{G}(x^{k+1})\| \le C_0
\|\mathcal{G}(x^k)\|^{2} \;\mbox{ whenever }\;k\ge k_0.
\end{equation*} Employing \eqref{thm3_eq2} and $\|\mathcal{G}(x^k)\|
\le (2+L_1)\,\mathrm{dist}(x^k;\mathcal{X}^*)$, which comes from
Proposition~\ref{bound_rxk}, gives us a positive number
$\tilde{C}_0$ such that the claimed condition \eqref{quad} holds.
This completes the proof of the theorem. \end{proof}

The concluding result of this section concerns the other kind of
metric $q$-subregularity of the subgradient mapping in
\eqref{prob_ori_1} in the case where $q>1$. As discussed in
Section~\ref{sec:2}, this type of higher-order metric subregularity
is rather new in the literature, and it has never been used in
applications to numerical optimization. The following theorem shows
that the higher-order subregularity assumption imposed on the
subgradient mapping $\partial F$ at the point in question allows us
to derive a  counterpart of Theorem~\ref{Newton_thm1} with
establishing the {\em convergence rate}, which may be {\em higher
than quadratic}. Indeed, Proposition~\ref{char_q}  characterizes the
metric $q$-subregularity of the subgradient mapping by an equivalent
$\frac{1+q}{q}$ growth condition for each $q>0$. Based on this
equivalence for $q > 1$, the imposed metric $q$-subregularity
implies a sharper growth behavior of $F$ around the solution point
in comparison with the quadratic growth. Consequently, the
convergence rate faster than the quadratic rate can be achieved.

\begin{theorem}\label{Newton_mqs} Let $\{x^k\}$ be the sequence
generated by Algorithm~{\rm\ref{Newton}} with
\revise{$\alpha_k=\min \left\{ \bar{\alpha}, c\|\mathcal{G}(x^k)\|^{\rho}\right\}$} as $\rho\in(0,1]$, and let
$\bar{x}\in \mathcal{X}^*$ be any limiting point of $\{x^k\}_{k\in
K_0}$, where $K_0$ is taken from \eqref{K}. In addition to
the standing assumptions, suppose that the mapping $\nabla
f(x)+\partial g(x)$ is metrically $q$-subregular at $(\bar{x},0)$
with $q>1$, that the Hessian $\nabla^2 f$ is locally Lipschitzian
around $\bar{x}$, that $\|B_k-\nabla^2f(x^k)\|=O(
\|\mathcal{G}(x^k)\|)$, and that $\varrho\ge q(1+\rho)-1$ in
\eqref{inexact_condition_1}. Then there exists $k_0$ such
that $t_k=1$ for all $k\ge k_0$ and that $\{x^k\}$
converges to the point $\bar{x}$ with the convergence rate
$q(1+\rho)$. The latter means that for some $k_0,C_0>0$ we have
\begin{equation}\label{quad+}
\mathrm{dist}(x^{k+1};\mathcal{X}^*)\le
C_0\,\mathrm{dist}(x^k;\mathcal{X}^*)^{q(1+\rho)}\;\mbox{ whenever
}\;k\ge k_0. \end{equation} \end{theorem} \begin{proof} It follows
from the imposed metric $q$-subregularity condition with a fixed
degree $q>1$ that \begin{equation}\label{thm1_eq_mqs}
\Sigma(p)\cap\mathbb{B}_{\epsilon_1}(\bar{x})\subset\mathcal{X}^*+\kappa_1\|p\|^q\mathbb{B}\;\mbox{
for some }\;\epsilon_1,\kappa_1>0 \end{equation} whenever $p\in\R^n$
is sufficiently close to the origin. Following the proof of
Theorem~\ref{Newton_thm1}, we arrive at the estimate of
$\|\mathcal{R}_k(\tilde{x}^k,x^k)\|$ in \eqref{assump_Rk} with some
constant $c_3>0$,  where $\tilde{x}^k:=\hat{x}^k-e_k$ while
$\mathcal{R}_k(\tilde{x}^k,x^k)$, $\hat x^k$, and $e_k$ are defined
and analyzed similarly to the case of Theorem~\ref{Newton_thm1}.
Then there exists $\epsilon_3 \in (0,1]$ such that
$\tilde{x}^k\in\mathbb{B}_{\epsilon_1}(\bar{x})$ when
$x^k\in\mathbb{B}_{\epsilon_3}(\bar{x})$. Since
$\tilde{x}^k\in\Sigma(\mathcal{R}_k(\tilde{x}^k\,x^k))$, we combine
this with \eqref{thm1_eq_mqs} and get the estimates \[
\mathrm{dist}(\tilde{x}^k;\mathcal{X}^*)\le\kappa_1\|\mathcal{R}_k(\tilde{x}^k
,x^k)\|^q\le \kappa_1 c_3^q \|\mathcal{G}(x^k)\|^{q(1+\rho)}
\le\kappa_1c_3^q(2+L_1)^q\mathrm{dist}(x^k;\mathcal{X}^*)^{q(1+\rho)}\;\mbox{
and} \] \begin{equation}\label{Newton_mqs:convergence_rate}
\begin{array}{ll}
\mathrm{dist}(\hat{x}^k;\mathcal{X}^*)&\le\mathrm{dist}(\tilde{x}^k;\mathcal{X}^*)+\|e_k\|\le \kappa_1c_3^q(2+L_1)^q\mathrm{dist}(x^k;\mathcal{X}^*)^{q(1+\rho)} + \nu\|\mathcal{G}(x^k)\|^{1+\varrho}\\
&\le(\kappa_1c_3^q(2+L_1)^q+\nu
(2+L_1)^{1+\varrho})\mathrm{dist}(x^k;\mathcal{X}^*)^{q(1+\rho)}\;\mbox{
whenever }\;x^k\in\mathbb{B}_{\epsilon_3}(\bar{x}). \end{array}
\end{equation} Employing the induction arguments as in the proof of
Theorem~\ref{Newton_thm1} yields the existence of a natural number
$k_0$ such that we have $k+1 \in K_0$, $t_k=1$,
$x^{k+1}=\hat{x}^{k}$, $x^{k+1}\in\mathbb{B}_{\epsilon_3}(\bar{x})$
when $k \ge k_0$, and that the sequence $\{x^k\}$ converges to
$\bar{x}$ as $k\to\infty$. Hence the second estimate in
\eqref{Newton_mqs:convergence_rate} gives a positive number $C_0$
and a natural number $k_0$, which ensure the fulfillment the claimed
convergence rate \eqref{quad+} and thus complete the proof.
\end{proof}\vspace*{-0.25in}

\section{Superlinear Local Convergence with Non-Lipschitzian Hessians}\label{sec:6}\vspace*{-0.05in}

As seen in Section \ref{sec:5}, the imposed local Lipschitz
continuity of the Hessian mapping $\nabla^2f$ plays a crucial role
in the justifications of the local convergence results obtained
therein. In this section we show that the latter assumption can be
dropped with preserving a local superlinear convergence of
Algorithm~\ref{Newton} for a rather broad class of loss functions
$f$ that naturally appear in many practical models arising in
machine learning and statistics, which includes, e.g., linear
regression, logistic regression, and Poisson
regression.\vspace*{0.03in}

The class of loss functions $f$ of our consideration in this section satisfies the following structural properties.

\begin{assumption}\label{assum_sc_h} The loss function $f\colon\mathbb{R}^n\to\oR$ of \eqref{prob_ori_1} is represented in the form
\begin{equation}\label{f-str}
f(x):=h(Ax)+\langle b,x\rangle,
\end{equation}
where $A$ is an $m\times n$ matrix, $b\in\R^n$, and $h\colon\mathbb{R}^m\to\oR$ is a proper, convex, and l.s.c.\ function such that:
\begin{itemize}
\item[{\bf(i)}] $h$ is strictly convex on any compact and convex subset of the domain $\dom h$.

\item[{\bf(ii)}] $h$ is continuously differentiable on the set $\dom h$, which is assumed to be open, and the gradient mapping $\nabla h$ is Lipschitz continuous
on any compact subset $\Omega\subset\dom h$.
\end{itemize}
\end{assumption}

Due to the strict convexity of $h$, the linear mapping $x\rightarrow Ax$ in \eqref{f-str} is invariant over the solution set $\mathcal{X}^*$ to \eqref{prob_ori_1}. This is the contents of the following
result taken from \cite[Lemma~2.1]{LT-1992}.

\begin{lemma}\label{invariantSC} Under the fulfillment of Assumption~{\rm\ref{assum_sc_h}} there exists $\bar{y}\in\mathbb{R}^m$ such that $Ax=\bar{y}$ for all $x\in\mathcal{X}^*$.
\end{lemma}

The next lemma is a counterpart of Lemma~\ref{holder_d_k} without imposition the local Lipschitz continuity of the Hessian $\nabla^2f$. By furnishing a similar while somewhat different scheme in
comparison with Lemma~\ref{holder_d_k}, we establish new direction estimates of Algorithm~\ref{Newton} used in what follows. Note that we do not exploit in the lemma the structural conditions on $f$
listed in Assumption~\ref{assum_sc_h}.

\begin{lemma}\label{Newton_lem2} Let $\{x^k\}$ be the sequence generated by Algorithm~{\rm\ref{Newton}} with \revise{$\alpha_k=\min \left\{ \bar{\alpha}, c\|\mathcal{G}(x^k)\|^{\rho}\right\}$} and $\rho\in(0,1)$, and let $\bar{x}\in\mathcal{X}^*$
be any limiting point of $\{x^k\}$. In addition to Assumption~{\rm\ref{assumptions_f}} and \eqref{assum_B_bounded}, suppose that the Hessian mapping $\nabla^2 f$ is continuous at $\bar{x}\in\mathcal{X}^*$, that $\|B_k-\nabla^2 f(x^k)\|\to 0$ as $k\rightarrow\infty$, and that the subgradient mapping $\nabla f(x)+\partial g(x)$ is metrically subregular at $(\bar{x},0)$. Then given an arbitrary quantity $\delta>0$, there exist $\epsilon>0$ and $k_0\in\N$ such that for $d^k:=\hat x^k-x^k$ we have the estimates
\begin{equation}\label{dk2}
\alpha_k\|d^k\|\le\delta\,\mathrm{dist}(x^k;\mathcal{X}^*)\;\mbox{ and }\;\|d^k\|\le\delta\,\mathrm{dist}(x^k;\mathcal{X}^*)^{1-\rho}\;\mbox{ when }\;x^k\in \mathbb{B}_{\epsilon}(\bar{x})\;
\mbox{ and }\;k>k_0.
\end{equation}
\end{lemma}
\begin{proof} Since $\hat{x}^k$ is an inexact solution to subproblem \eqref{subproblem} satisfying \eqref{inexact_condition_1}, we get by Lemma~\ref{lem1} that there exists $e_k$ for which both
conditions in \eqref{ek} hold. Taking the projection $\pi_x^k$ of $x^k$ onto the solution set ${\cal X}^*$ and arguing as in the proof of Lemma~\ref{holder_d_k} bring us to the inequality in
\eqref{mon2}, which together with the direction estimate in \eqref{dist-est2} ensures that
\begin{equation}\label{dir-est}
\begin{aligned}
\alpha_k\|d^k\|&\le\Big(\|\nabla f(x^k)+\nabla^2 f(x^k)(\pi_x^k-x^k)-\nabla f(\pi_x^k)\|+\|B_k-\nabla^2f(x^k)\|\mathrm{dist}(x^k;\mathcal{X}^*)\\
&\quad+2\alpha_k\,\mathrm{dist}(x^k;\mathcal{X}^*)+(1+M+\alpha_k)\nu\|\mathcal{G}(x^k)\|^{1+\varrho}\Big).
\end{aligned} \end{equation} It follows from the mean value theorem
and the choice of $\pi_x^k$ as the projection of $x^k$ onto ${\cal
X}^*$ that \begin{equation*} \begin{aligned}
\|\nabla f(x^k)+\nabla^2 f(x^k)(\pi_x^k-x^k)-\nabla f(\pi_x^k)\|&=\|(\nabla^2 f(x^k)-\nabla^2 f(\xi^k))(\pi_x^k-x^k)\|\\
&\le\|\nabla^2f(x^k)-\nabla^2f(\xi^k)\|\mathrm{dist}(x^k;\mathcal{X}^*),
\end{aligned} \end{equation*} where
$\xi^k:=\lambda^kx^k+(1-\lambda^k)\pi_x^k$ for some
$\lambda^k\in(0,1)$, and hence $\xi^k\to\bar{x}$ when
$x^k\to\bar{x}$ as $k\to\infty$. Then passing to the limit as
$k\to\infty$ and using the assumed continuity of $\nabla^2f$ at
$\ox$ show that $\|\nabla^2 f(x^k)- \nabla^2 f(\xi^k)\|\rightarrow
0$. Since \revise{$\alpha_k=\min \left\{ \bar{\alpha}, c\|\mathcal{G}(x^k)\|^{\rho}\right\} \rightarrow 0$} and
$\|\mathcal{G}(x^k)\|\le(2+L_1)\mathrm{dist}(x^k;\mathcal{X}^*)$ by
Proposition~\ref{bound_rxk}, and since $\|B_k-\nabla^2f(x^k)\|\to 0$
as $k\rightarrow\infty$, for any $\delta>0$ we find $\epsilon>0$ and
$k_0\in\N$ such that \[
\alpha_k\|d^k\|\le\delta\,\mathrm{dist}(x^k;\mathcal{X}^*)\;\mbox{
when }\;x^k\in\mathbb{B}_{\epsilon}(\bar{x})\;\mbox{ and }\;k>k_0,
\] which justifies the first estimate in \eqref{dk2}. To verify
finally the second one in \eqref{dk2}, employ
Proposition~\ref{bound_dist} and the above expression of $\alpha_k$
to find positive numbers $\epsilon_1$ and $c_1$ ensuring the
inequality \[ \alpha_k\ge
c_1\,\mathrm{dist}(x^k;\mathcal{X}^*)^{\rho}\;\mbox{ for all
}\;x\in\mathbb{B}_{\epsilon_1}(\bar{x}). \] Combining the latter
with the first estimate in \eqref{dk2} tells us that for the fixed
number $\dd>0$ there exist $\ve>0$ and $k>k_0$ such that the second
estimate in \eqref{dk2} is also satisfied, and thus the proof is
complete. \end{proof}

By the same arguments as in the proof of Lemma~\ref{K0}, we can also
show that the set $K_0$ defined in \eqref{K} is infinite in the
setting of Lemma \ref{Newton_lem2}. Now we are ready to derive the
promised result showing that the sequence of iterates, which are
generated by Algorithm~\ref{Newton} for the structured problem
\eqref{prob_ori_1} considered in this section, converges
superlinearly to a given optimal solution $\ox\in{\cal X}^*$ without
the local Lipschitz continuity assumption on the Hessian mapping
$\nabla^2f$.

\begin{theorem}\label{thm-nonlip} Let $\{x^k\}$ be the sequence of iterates generated by Algorithm~{\rm\ref{Newton}} with \revise{$\alpha_k=\min \left\{ \bar{\alpha}, c\|\mathcal{G}(x^k)\|^{\rho}\right\}$} and $\rho\in(0,1)$,
and let $\bar{x}\in\mathcal{X}^*$ be any limiting point of the sequence $\{x^k\}_{k\in K_0}$ with the set $K_0$ defined in \eqref{K}. Suppose in addition to the assumptions of
Lemma~{\rm\ref{Newton_lem2}} that the loss function $f$ is given in the structured form \eqref{f-str} under the fulfillment of Assumption~{\rm\ref{assum_sc_h}}, and that at each iteration step $k$
the matrix $B_k$ is represented in the form $B_k=A^TD_kA$, where $A$ is taken from \eqref{f-str} while $D_k\in\mathbb{R}^{m\times m}$ is some positive semidefinite matrix. Then there exists a natural
number $k_0$ such that $t_k=1$ for all $k\ge k_0$, and that the sequence $\{x^k\}$ converges to $\bar{x}$ with the superlinear convergence rate, i.e., there is $k_1$ for which we have
\begin{equation}\label{rate}
\mathrm{dist}(x^{k+1};\mathcal{X}^*)=o\big(\mathrm{dist}(x^k;\mathcal{X}^*)\big)\;\mbox{ whenever }\;k\ge k_1.
\end{equation}
\end{theorem}
\begin{proof} Proceeding similarly to the proof of Theorem~\ref{Newton_thm1}, at each iteration step $k$ we have the vector ${\cal R}_k(\tilde x^k,x^k)$ defined in \eqref{Rk} with
$\tilde{x}^k:=\hat{x}^k-e_k$, where $\hat{x}^k$ is an inexact solution of \eqref{subproblem} satisfying \eqref{inexact_condition_1}, and where $e_k$ is taken from \eqref{ek}. These relationships and the
mean value theorem applied to the gradient mapping $\nabla f$ on $[x^k,\tilde x^k]$ give us a vector $\tilde{\xi}^k:=\tilde{\lambda}^kx^k+(1-\tilde{\lambda}^k)\tilde{x}^k$ with some
$\tilde{\lambda}^k\in(0,1)$ such that
\begin{equation*}
\begin{aligned}
\|\mathcal{R}_k(\tilde{x}^k, x^k)\|&=\|\nabla f(\tilde{x}^k)-\nabla f(x^k)-H_k(\tilde{x}^k-x^k)+e_k-H_ke_k\|\\
&=\|\nabla f(\tilde{x}^k)-\nabla f(x^k)-(B_k+\alpha_kI)(\tilde{x}^k-x^k)+e_k-H_ke_k\|\\
&\le\|\nabla f(\tilde{x}^k)-\nabla f(x^k)-\nabla^2f(x^k)(\tilde{x}^k-x^k)\|+\|B_k-\nabla^2f(x^k)\|\cdot\|\tilde{x}^k-x^k\|\\
&\quad+\alpha_k\|\tilde{x}^k-x^k\|+(1+M)\|e_k\|\\
&\le\|(\nabla^2f(\tilde{\xi}^k)-\nabla^2 f(x^k))(\tilde{x}^k-x^k)\|+\|(B_k-\nabla^2f(x^k))(\tilde{x}^k-x^k)\|\\
&\quad+\alpha_k\|d^k\|+(1+M)\nu\|\mathcal{G}(x^k)\|^{1+\varrho}.
\end{aligned}
\end{equation*}
Let $\tilde{\pi}_x^k$ and $\pi_x^k$ be the projections of $\tilde{x}^k$ and $x^k$ onto $\mathcal{X}^*$, respectively. Then it follows from Lemma~\ref{invariantSC} that $A\tilde{\pi}_x^k=A\pi_x^k$.
By Assumption~\ref{assum_sc_h} we have $\nabla^2f(x)=A^T\nabla^2h(x)A$, and thus
\begin{equation*}
\big(\nabla^2f(\tilde{\xi}^k)-\nabla^2f(x^k)\big)(\tilde{x}^k-x^k)=\big(\nabla^2f(\tilde{\xi}^k)-\nabla^2f(x^k)\big)(\tilde{x}^k-\tilde{\pi}_x^k+\pi_x^k-x^k).
\end{equation*}
Using the assumed representation $B_k=A^TD_kA$ of the matrix $B_k$, we similarly get that
\begin{equation*}
\big(B_k-\nabla^2f(x^k)\big)(\tilde{x}^k-x^k)=\big(B_k-\nabla^2f(x^k)\big)(\tilde{x}^k-\tilde{\pi}_x^k+\pi_x^k-x^k).
\end{equation*}
Plugging the obtained expressions into the above estimate of $\|{\cal R}_k\|$ gives us
\begin{equation*}
\begin{aligned}
\|\mathcal{R}_k(\tilde{x}^k,x^k)\|
&\le\|(\nabla^2f(\tilde{\xi}^k)-\nabla^2f(x^k))(\tilde{x}^k-\tilde{\pi}_x^k+\pi_x^k-x^k)\|+\|(B_k-\nabla^2f(x^k))(\tilde{x}^k-\tilde{\pi}_x^k+\pi_x^k-x^k)\|\\
&\quad+\alpha_k\|d^k\|+(1+M)\nu\|\mathcal{G}(x^k)\|^{1+\varrho}\\
&\le\|\nabla^2f(\tilde{\xi}^k)-\nabla^2f(x^k)\|\big(\mathrm{dist}(\tilde{x}^k;\mathcal{X}^*)+\mathrm{dist}(x^k;\mathcal{X}^*)\big)\\
&\quad+\|B_k-\nabla^2f(x^k)\|\big(\mathrm{dist}(\tilde{x}^k;\mathcal{X}^*)+\mathrm{dist}(x^k;\mathcal{X}^*)\big)+\alpha_k\|d^k\|+(1+M)\nu(2+L_1)^{\varrho}\mathrm{dist}(x^k;\mathcal{X}^*)^{1+\varrho}.
\end{aligned} \end{equation*} It follows from the second estimate of
Lemma~\ref{Newton_lem2} that $\|d^k\|\rightarrow 0$ as
$k\rightarrow\infty$ and $x^k\rightarrow\bar{x}$. Since
$x^k\rightarrow\bar{x}$ implies that $\tilde{x}^k\rightarrow
\bar{x}$ as $k\rightarrow\infty$, the assumed continuity of
$\nabla^2f$ at $\ox$ and the above construction of $\tilde{\xi}^k$
tell us that $\|\nabla^2f(\tilde{\xi}^k)-\nabla^2f(x^k)\|\rightarrow
0$ as $k\rightarrow\infty$ and $x^k\rightarrow\bar{x}$. Now the
first estimate of Lemma~\ref{Newton_lem2} ensures that
$\alpha_k\|d^k\|=o(\mathrm{dist}(x^k;\mathcal{X}^*))$ as
$k\rightarrow\infty$ and $x^k\rightarrow\bar{x}$. Combining this
with $\|B_k-\nabla^2f(x^k)\|\rightarrow 0$ as $k\rightarrow\infty$
allows us to conclude that for any $\delta>0$ there exist a positive
number $\epsilon$ and a natural number $k_0$ such that
\begin{equation}\label{rk1}
\|\mathcal{R}_k(\tilde{x}^k,x^k)\|\le\delta\left(\mathrm{dist}(\tilde{x}^k;\mathcal{X}^*)+\mathrm{dist}(x^k;\mathcal{X}^*)\right)\;\mbox{
whenever }\;x^k\in\mathbb{B}_{\epsilon}(\bar{x})\; \mbox{ and
}\;k>k_0. \end{equation} It follows from the metric subregularity
assumption that we have inclusion \eqref{H_error_bound_eq} with $q =
1$ and the perturbed solution map $\Sigma(p)$ therein. Since
$\tilde{x}^k\in\Sigma(\mathcal{R}_k(\tilde{x}^k,x^k))$ as shown
above, there is $\kappa_1>0$ with \begin{equation*}
\mathrm{dist}(\tilde{x}^k;\mathcal{X}^*)\le\kappa_1\|\mathcal{R}_k(\tilde{x}^k,x^k)\|\le\kappa_1\delta\left(\mathrm{dist}(\tilde{x}^k;\mathcal{X}^*)+\mathrm{dist}(x^k;\mathcal{X}^*)\right)\;
\mbox{ for all }\;x^k\in\mathbb{B}_{\epsilon}(\bar{x})\;\mbox{ and
}\;k>k_0, \end{equation*} which implies that
$\mathrm{dist}(\tilde{x}^k;\mathcal{X}^*)=o(\mathrm{dist}(x^k;\mathcal{X}^*))$
as $k\to\infty$. Recalling the estimates \begin{equation*}
\mathrm{dist}(\hat{x}^k;\mathcal{X}^*)\le\mathrm{dist}(\tilde{x}^k;\mathcal{X}^*)+\|e_k\|\;\mbox{
and }\;\|e_k\|\le\nu\|\mathcal{G}(x^k)\|^{1+\varrho}\le\nu
(2+L_1)^{1+
\varrho}\mathrm{dist}(x^k;\mathcal{X}^*)^{1+\varrho}=o\big(\mathrm{dist}(x^k;\mathcal{X}^*)\big),
\end{equation*} we readily get, for all the numbers
$\delta,\ve,k_0,k$ taken from \eqref{rk1}, the conditions
\begin{equation}\label{Newton_thm2_eq4}
\mathrm{dist}(\hat{x}^k;\mathcal{X}^*)=o\big(\mathrm{dist}(x^k;\mathcal{X}^*)\big)\;\mbox{
and
}\;\mathrm{dist}(\hat{x}^k;\mathcal{X}^*)\le\delta\,\mathrm{dist}(x^k;\mathcal{X}^*),
\end{equation} which ensure therefore the existence of a positive
number $\ve_0$ and a natural number $k_0$  such that
\begin{equation}\label{Newton_thm2_eq3}
\mathrm{dist}(\hat{x}^k;\mathcal{X}^*)\le\frac{\sigma}{(2+L_1)\kappa_2}\mathrm{dist}(x^k;\mathcal{X}^*)\;\mbox{
whenever }\;x^k\in\mathbb{B}_{\epsilon_0}(\bar{x})\;\mbox{ and
}\;k>k_0. \end{equation} Employing Lemma~\ref{Newton_lem2}, suppose
without loss of generality that there exists $c_1>0$ with
\begin{equation}\label{thm2_eq5} \|d^k\|\le
c_1\,\mathrm{dist}(x^k;\mathcal{X}^*)^{1-\rho}\;\mbox{ for all
}\;\;x^k\in\mathbb{B}_{\epsilon_0}(\bar{x})\;\mbox{ and }\;k>k_0.
\end{equation} Since $C>F(x^0)\ge F^*$ in our algorithm, and since
$F$ is continuous on $\dom F$, let us pick $\epsilon_0>0$ to be so
small that condition \eqref{thm_eq2} holds. Defining the positive
numbers \begin{equation}\label{sigma}
\tilde{\sigma}:=\frac{\sigma}{(2+L_1)\kappa_2}<1\;\mbox{ and
}\;\tilde{\epsilon}:=\min\bigg\{\frac{\epsilon_0}{2},\Big(\frac{1-\tilde{\sigma}^{1-\rho}}{2
c_1}\epsilon_0\Big)^{\frac{1}{1-\rho}}\bigg\} \end{equation} and
invoking the set $K_0$ from \eqref{K}, we intend to show that if
$x^{k_1}\in\mathbb{B}_{\tilde{\epsilon}}(\bar{x})$ with $k_1>k_0$
and $k_1\in K_0$, then \begin{equation}\label{k1} k+1\in
K_0,\;t_k=1,\;x^{k+1}=\hat{x}^{k},\;\mbox{ and
}\;x^{k+1}\in\mathbb{B}_{\epsilon_0}(\bar{x})\;\mbox{ whenever
}\;k\ge k_1. \end{equation} To prove it by induction, observe first
that for $k:=k_1$ all the conditions in \eqref{k1} can be verified
similarly to the proof of \eqref{k0} in Theorem~\ref{Newton_thm1}
with the replacement of $k_0$ by $k_1$. Considering now the general
case where $k>k_1$ in \eqref{k1}, suppose that the latter holds for
any $k-1\ge\ell\ge k_1$, which clearly yields
$\mathrm{dist}(x^{\ell+1};\mathcal{X}^*)
\le\tilde{\sigma}\,\mathrm{dist}(x^{\ell};\mathcal{X}^*)$. Then the
above estimates and the parameter choice in \eqref{sigma} ensure
that \begin{equation}\label{thm2_eq3} \begin{aligned}
\|\hat{x}^k-x^{k_1}\|&\le\sum_{\ell=k_1}^{k}\|d^{\ell}\|\le\sum_{\ell=k_1}^{k}c_1\,\mathrm{dist}(x^{\ell};\mathcal{X}^*)^{1-\rho}\\
&\le\sum_{\ell=k_1}^{k}c_1\tilde{\sigma}^{(1-\rho)(\ell-k_1)}\mathrm{dist}(x^{k_1};\mathcal{X}^*)^{1-\rho}\le\frac{c_1}{1-\tilde{\sigma}^{1-\rho}}\mathrm{dist}(x^{k_1};\mathcal{X}^*)^{1-\rho}\\
&\le\frac{c_1}{1-\tilde{\sigma}^{1-\rho}}\|x^{k_1}-\bar{x}\|^{1-\rho},
\end{aligned} \end{equation} where the second inequality follows
from \eqref{thm2_eq5}. Thus by \eqref{sigma} and \eqref{thm2_eq3} we
have \[
\|\hat{x}^k-\bar{x}\|\le\|\hat{x}^k-x^{k_1}\|+\|x^{k_1}-\bar{x}\|\le\frac{c_1}{1-\tilde{\sigma}^{1-\rho}}\|x^{k_1}-\bar{x}\|^{1-\rho}+\|x^{k_1}-\bar{x}\|\le\epsilon_0,
\] which readily implies, similarly to the case where $k=k_1$, the
fulfillment of \eqref{k1} for any $k\ge k_1$. Moreover, remembering
that $\bar{x}$ is a limiting point of $\{x^k\}_{k\in K_0}$ and using
\eqref{k1} together with \eqref{thm2_eq3} allow us to check that for
any $\tilde{k}\ge k_1$ we have \[
\|x^k-\bar{x}\|\le\frac{c_1}{1-\tilde{\sigma}^{1-\rho}}\|x^{\tilde{k}}-\bar{x}\|^{1-\rho}+\|x^{\tilde{k}}-\bar{x}\|\;\mbox{
whenever }\;k>\tilde{k}. \] Further, let $\tilde x$ be any limiting
point of the original iterative sequence $\{x^k\}$. Then the passage
to the limit in the above inequality as $k\to\infty$ gives us \[
\|\tilde{x}-\bar{x}\|\le\frac{c_1}{1-\tilde{\sigma}^{1-\rho}}\|x^{\tilde{k}}-\bar{x}\|^{1-\rho}+\|x^{\tilde{k}}-\bar{x}\|\;\mbox{
for all }\;\tilde{k}\ge k_1. \] Passing finally the limit as $\tilde
k\to\infty$ in the latter inequality and recalling that $\bar{x}$ is
a limiting point of $\{x^k\}_{k \in K_0}$ tell us that
$\|\tilde{x}-\bar{x}\|=0$, which verifies therefore that $\{x^k\}$
converges to $\bar{x}$ as $k\to\infty$. The claimed estimate
\eqref{rate} of the convergence rate follows now from
\eqref{Newton_thm2_eq4}, and this completes the proof of the
theorem. \end{proof}

To conclude this section, observe that the standard choice of
$B_k=\nabla^2 f(x^k)$ in Algorithm~\ref{Newton} clearly implies that
the assumed representation $B_k=A^TD_kA$ and the condition
$\|B_k-\nabla^2 f(x^k)\|\rightarrow 0$ as $k\rightarrow\infty$ {\em
hold automatically} due to $\nabla^2 f(x^k)=A^T\nabla^2h(Ax^k)A$ and
the positive semidefiniteness of the Hessian $\nabla^2h(Ax^k)$ under
Assumption~\ref{assum_sc_h} on the loss function $f$ imposed
here.\vspace*{-0.15in}

\section{Numerical Experiments for Regularized Logistic Regression}\label{sec:7}\vspace*{-0.05in}

In the last section of the paper we conduct numerical experiments on
solving the $l_1$ regularized logistic regression problem to support
our theoretical results and compare them with the numerical
algorithm from \cite{yue2019family} applicable to this problem. All
of the numerical experiments are implemented on a laptop with
Intel(R) Core(TM) i7-9750H CPU@ 2.60GHz and 32.00 GB memory. All the
codes are written in MATLAB 2021a.

Supposing that we have some given training data pairs
$(a_i,b_i)\in\mathbb{R}^n\times\{-1,1\}$ as $i=1,\ldots,N$, the
optimization problem for $l_1$ regularized logistic regression is
defined by \begin{equation}\label{logistic_regression}
\min_{x}\frac{1}{N}\sum_{i=1}^{N}\log(1+\exp(-b_ix^Ta_i))+\lambda\|x\|_1,
\end{equation} where the regularization term $\|x\|_1$ promotes
sparse structures on solutions, and where $\lambda>0$ is the
regularization parameter balancing sparsity and fitting error.
Problem \eqref{logistic_regression} is a special case of
\eqref{prob_ori_1} with
$f(x):=\frac{1}{N}\sum_{i=1}^{N}\log(1+\exp(-b_ix^Ta_i))$ and
$g(x):=\lambda\|x\|_1$. In all the experiments, the matrix $B_k$ in
our proximal Newton-type Algorithm~\ref{Newton} is chosen as the
Hessian matrix of $f$ at the iterate $x^k$, i.e., $B_k:=\nabla^2
f(x^k)$. We set $\nu:=0.9$ and $\varrho:= \rho$ in the inexact
conditions \eqref{inexact_condition_1} for determining an inexact
solution $\hat{x}^k$ to subproblem~\eqref{subproblem}. We also set
$\theta:=0.1$, $\sigma:=0.5$, $\gamma:=0.5$, $C:=2F(x^0)$, $\bar{\alpha}:= 10^{-4}$ and $c :=
10^{-8}$, and then test the three values $\{0.1, 0.5, 1\}$ of
parameter $\rho$ in Algorithm~\ref{Newton}. As shown in
\cite[Theorem~8]{FBconvex}, the subgradient mapping $\nabla
f(x)+\partial g(x)$ is metrically subregular at $(\bar{x},0)$ for
any $\bar{x}\in\mathcal{X}^*$. It can be easily verified that all
the assumptions required in Theorem~\ref{Newton_thm1} are satisfied,
and hence the sequence of iterates generated by the proposed
algorithm for the tested problem \eqref{logistic_regression} locally
converges to the prescribed optimal solution with a
superlinear/quadratic convergence rate.

Here we test four real datasets ``colon-cancer", ``duke
breast-cancer", ``leukemia" and ``rcv1\_train.binary" downloaded
from the SVMLib repository \cite{libsvm}\footnote{
\url{http://www.csie.ntu.edu.tw/~cjlin/libsvmtools/datasets/}.}, and
their sizes are given in Table \ref{data_description}, where
nnz($A$) denotes the number of nonzero elements of the feature
matrix $A$. All of these tested real datasets have more columns than
rows, and hence the loss function $f$ in the corresponding problem
\eqref{logistic_regression} is not strongly convex. We use
$\mathtt{normr}$ in Matlab to normalize the rows of each dataset to
make them have unit norm. \begin{table} \centering \caption{Tested
datasets} \label{data_description} \begin{tabular}{l l l l l l l l l
} \hline Dataset & \hspace{15pt} & Data points(N) &\hspace{10pt} &
Features(n) & \hspace{10pt} & nnz($A$) & \hspace{10pt} &
Density($A$)(\%) \\ \Xhline{1.2pt} \rule{0pt}{10pt} colon-cancer & &
62 & & 2000 && 124000 && 100 \\  \rule{0pt}{4pt}
duke breast-cancer & & 44 && 7129 && 313676 && 100 \\
\rule{0pt}{4pt} leukemia && 38 &&  7129 && 270902 && 100 \\
\rule{0pt}{4pt}
rcv1\_train.binary && 20,242  && 47,236 && 1498952  && 0.16  \\
\hline \end{tabular} \end{table}

Since the IRPN algorithm proposed in \cite{yue2019family} does not
require $f$ in \eqref{logistic_regression} to be strongly convex and
problem \eqref{logistic_regression} satisfies all the assumptions
required by IRPN, we are going to compare our proposed proximal
Newton-type Algorithm~\ref{Newton} with IRPN. Note that the IRPN
code is collected from \url{https://github.com/ZiruiZhou/IRPN} with
some modifications to match the objective function in
\eqref{logistic_regression}. We set {$\theta=\beta:=0.25$},
$\zeta:=0.4$, $\eta:=0.5$ and $c := 10^{-6}$ in IRPN as suggested in
\cite{yue2019family}. Also we set {$\rho:= 0 $ and $0.5$}, since
these two specifications of IRPN perform best as shown in
\cite{yue2019family}. It should be noticed that in such setting both
our Algorithm~\ref{Newton} and IRPN require solving subproblem
\eqref{subproblem} at each iteration. This subproblem can be solved
by the coordinate gradient descent method, which is implemented in
MATLAB as a C source MEX-file.\footnote{The code is downloaded from
\url{https://github.com/ZiruiZhou/IRPN}.}

We also compare our proposed proximal Newton-type
Algorithm~\ref{Newton} with the proximal-Newton method PNOPT
(Proximal Newton OPTimizer) proposed in \cite{Lee2014}, although the
theoretical result of PNOPT in \cite{Lee2014} requires the strong
convexity of $f$ in \eqref{logistic_regression}. Note that the PNOPT
code is collected from
\url{http://web.stanford.edu/group/SOL/software/pnopt/} with
replacing the subproblem solver by the coordinate gradient descent
method mentioned \revise{above for a fair comparison}.

The initial points in all experiments are set to be a zero vector.
Each method is terminated at the iterate $x^k$ if the accuracy
$\mathtt{TOL}$ is reached by $\|\mathcal{G}(x^k)\|\le\mathtt{TOL}$
with the residual $\|\mathcal{G}(x^k)\|$ defined via the
prox-gradient mapping \eqref{prox-gra}. The maximum number of
iterations for the coordinate gradient descent method to solve the
corresponding subproblems is $10000$.

The achieved numerical results are presented in
Tables~\ref{table_colon}, \ref{table_leu}, \ref{table_duk}, and
\ref{table_rcv1}. 
\begin{table} 
	\centering \caption{Numerical comparison on
		colon-cancer dataset with $\lambda = 10^{-4}$ and  $\lambda =
		10^{-6}$} \label{table_colon} \resizebox{\textwidth}{!}{
		\begin{tabular}{l  l || r r r r r r || c  c  c  c  c c }
			\hline 
			& & \multicolumn{6}{c||}{$\lambda = 10^{-4}$} & \multicolumn{6}{c}{$\lambda = 10^{-6}$} \\ \cline{3-14} 
			$\mathtt{TOL}$ & Solver & PNOPT & IRPN & IRPN & Algorithm \ref{Newton} & Algorithm \ref{Newton}& Algorithm \ref{Newton} & PNOPT & IRPN & IRPN & Algorithm \ref{Newton} & Algorithm \ref{Newton} & Algorithm \ref{Newton} \\ 
			& & & $\rho = 0$ & $\rho = 0.5$ & $\rho = 0.1$ & $\rho = 0.5$ & $\rho = 1$ & & $\rho = 0$ & $\rho = 0.5$ & $\rho = 0.1$ & $\rho = 0.5$ & $\rho = 1$ \\ \Xhline{1.2pt} \rule{0pt}{10pt}& Outer iterations & 4 & 4 & 4 & 4 & 4 & 4 &5 & 5 & 5 & 5 & 5 & 5 \\ 
			$10^{-3}$  & Inner iterations & 22 & 8 & 61 & 8 & 48 & 210 & 10 & 11 & 13 & 10 & 10 & 889 \\ 
			& Time(s) & 0.01 & 0.01 & 0.02 & 0.01 & 0.02 & 0.08 & 0.00 & 0.01 & 0.01 & 0.01 & 0.01 & 0.32 \\ \hline 
			& Outer iterations & 5 & 7 & 5 & 7 & 5 & 5 &7 & 7 & 7 & 7 & 7 & 7 \\ 
			$10^{-4}$  & Inner iterations & 54 & 29 & 127 & 42 & 107 & 319 & 15 & 17 & 231 & 14 & 136 & 1894 \\ 
			& Time(s) & 0.02 & 0.01 & 0.04 & 0.03 & 0.05 & 0.12 & 0.01 & 0.01 & 0.08 & 0.01 & 0.05 & 0.66 \\ \hline 
			& Outer iterations & 7 & 12 & 6 & 9 & 6 & 6 &9 & 14 & 9 & 9 & 9 & 9 \\ 
			$10^{-5}$  & Inner iterations & 103 & 74 & 188 & 89 & 150 & 442 & 67 & 31 & 464 & 23 & 357 & 3164 \\ 
			& Time(s) & 0.04 & 0.03 & 0.06 & 0.05 & 0.06 & 0.16 & 0.02 & 0.02 & 0.15 & 0.01 & 0.12 & 1.07 \\ \hline 
			& Outer iterations & 8 & 24 & 7 & 11 & 7 & 7 &10 & 156 & 10 & 11 & 10 & 10 \\ 
			$10^{-6}$  & Inner iterations & 134 & 155 & 264 & 118 & 200 & 603 & 113 & 315 & 731 & 88 & 580 & 3680 \\ 
			& Time(s) & 0.05 & 0.06 & 0.09 & 0.06 & 0.08 & 0.22 & 0.04 & 0.16 & 0.23 & 0.03 & 0.19 & 1.25 \\ \hline 
			& Outer iterations & 9 & 43 & 7 & 12 & 7 & 7 &12 & 878 & 11 & 13 & 11 & 11 \\ 
			$10^{-7}$  & Inner iterations & 176 & 264 & 264 & 135 & 200 & 603 & 258 & 1759 & 1025 & 168 & 813 & 4226 \\ 
			& Time(s) & 0.07 & 0.10 & 0.09 & 0.06 & 0.08 & 0.22 & 0.08 & 0.92 & 0.33 & 0.06 & 0.27 & 1.42 \\ \hline 
			& Outer iterations & 10 & 66 & 8 & 13 & 8 & 8 &13 & 3137 & 13 & 18 & 12 & 12 \\ 
			$10^{-8}$  & Inner iterations & 229 & 379 & 410 & 153 & 334 & 895 & 294 & 6277 & 1526 & 401 & 1049 & 4770 \\ 
			& Time(s) & 0.09 & 0.15 & 0.14 & 0.07 & 0.14 & 0.33 & 0.09 & 3.22 & 0.49 & 0.13 & 0.34 & 1.59 \\ \hline \end{tabular}} \end{table}

\begin{table}
	\centering
	\caption{Numerical comparison on leukemia dataset with $\lambda =
		10^{-4}$ and $\lambda = 10^{-6}$} \label{table_leu}
	\resizebox{\textwidth}{!}{
		\begin{tabular}{l  l || r r r r r r || c  c  c  c  c c }
			\hline 
			& & \multicolumn{6}{c||}{$\lambda = 10^{-4}$} & \multicolumn{6}{c}{$\lambda = 10^{-6}$} \\ \cline{3-14} 
			$\mathtt{TOL}$ & Solver & PNOPT & IRPN & IRPN & Algorithm \ref{Newton} & Algorithm \ref{Newton}& Algorithm \ref{Newton} & PNOPT & IRPN & IRPN & Algorithm \ref{Newton} & Algorithm \ref{Newton} & Algorithm \ref{Newton} \\ 
			& & & $\rho = 0$ & $\rho = 0.5$ & $\rho = 0.1$ & $\rho = 0.5$ & $\rho = 1$ & & $\rho = 0$ & $\rho = 0.5$ & $\rho = 0.1$ & $\rho = 0.5$ & $\rho = 1$ \\ \Xhline{1.2pt} \rule{0pt}{10pt}& Outer iterations & 3 & 3 & 3 & 3 & 3 & 3 &5 & 5 & 5 & 5 & 5 & 5 \\ 
			$10^{-3}$  & Inner iterations & 8 & 8 & 22 & 6 & 14 & 89 & 10 & 12 & 10 & 10 & 10 & 561 \\ 
			& Time(s) & 0.01 & 0.01 & 0.02 & 0.01 & 0.01 & 0.07 & 0.01 & 0.02 & 0.02 & 0.01 & 0.01 & 0.40 \\ \hline 
			& Outer iterations & 5 & 5 & 5 & 6 & 5 & 5 &7 & 7 & 7 & 7 & 7 & 7 \\ 
			$10^{-4}$  & Inner iterations & 56 & 22 & 123 & 31 & 116 & 286 & 15 & 19 & 207 & 14 & 122 & 1298 \\ 
			& Time(s) & 0.04 & 0.02 & 0.09 & 0.03 & 0.09 & 0.21 & 0.02 & 0.03 & 0.18 & 0.02 & 0.09 & 0.93 \\ \hline 
			& Outer iterations & 6 & 12 & 6 & 8 & 6 & 6 &8 & 11 & 8 & 8 & 8 & 8 \\ 
			$10^{-5}$  & Inner iterations & 85 & 55 & 231 & 67 & 219 & 511 & 34 & 28 & 351 & 16 & 219 & 1833 \\ 
			& Time(s) & 0.06 & 0.05 & 0.17 & 0.05 & 0.16 & 0.37 & 0.03 & 0.04 & 0.29 & 0.02 & 0.16 & 1.30 \\ \hline 
			& Outer iterations & 8 & 59 & 7 & 10 & 7 & 7 &10 & 181 & 10 & 12 & 9 & 9 \\ 
			$10^{-6}$  & Inner iterations & 164 & 243 & 382 & 118 & 344 & 794 & 124 & 368 & 584 & 51 & 316 & 2088 \\ 
			& Time(s) & 0.12 & 0.24 & 0.28 & 0.09 & 0.24 & 0.56 & 0.10 & 0.47 & 0.48 & 0.05 & 0.23 & 1.48 \\ \hline 
			& Outer iterations & 8 & 100 & 7 & 11 & 7 & 7 &12 & 1353 & 11 & 16 & 11 & 11 \\ 
			$10^{-7}$  & Inner iterations & 164 & 406 & 382 & 162 & 344 & 794 & 226 & 2712 & 864 & 152 & 812 & 3236 \\ 
			& Time(s) & 0.12 & 0.40 & 0.28 & 0.12 & 0.24 & 0.56 & 0.19 & 3.36 & 0.70 & 0.13 & 0.57 & 2.28 \\ \hline 
			& Outer iterations & 9 & 138 & 8 & 13 & 8 & 8 &13 & 6869 & 14 & 27 & 12 & 12 \\ 
			$10^{-8}$  & Inner iterations & 225 & 558 & 572 & 242 & 523 & 1177 & 316 & 13744 & 1519 & 403 & 1090 & 3783 \\ 
			& Time(s) & 0.16 & 0.56 & 0.41 & 0.18 & 0.37 & 0.82 & 0.26 & 18.93 & 1.21 & 0.32 & 0.76 & 2.66 \\ \hline 
\end{tabular}} \end{table}

\begin{table}
	\centering \caption{Numerical comparison on duke
		breast-cancer dataset with $\lambda = 10^{-4}$ and $\lambda =
		10^{-6}$} \label{table_duk} \resizebox{\textwidth}{!}{
		\begin{tabular}{l  l || r r r r r r || c  c  c  c  c c } \hline 
			& & \multicolumn{6}{c||}{$\lambda = 10^{-4}$} & \multicolumn{6}{c}{$\lambda = 10^{-6}$} \\ \cline{3-14} 
			$\mathtt{TOL}$ & Solver & PNOPT & IRPN & IRPN & Algorithm \ref{Newton} & Algorithm \ref{Newton}& Algorithm \ref{Newton} & PNOPT & IRPN & IRPN & Algorithm \ref{Newton} & Algorithm \ref{Newton} & Algorithm \ref{Newton} \\ 
			& & & $\rho = 0$ & $\rho = 0.5$ & $\rho = 0.1$ & $\rho = 0.5$ & $\rho = 1$ & & $\rho = 0$ & $\rho = 0.5$ & $\rho = 0.1$ & $\rho = 0.5$ & $\rho = 1$ \\ \Xhline{1.2pt} \rule{0pt}{10pt}& Outer iterations & 3 & 3 & 3 & 3 & 3 & 3 &5 & 5 & 5 & 5 & 5 & 5 \\ 
			$10^{-3}$  & Inner iterations & 10 & 7 & 25 & 6 & 15 & 86 & 10 & 12 & 12 & 10 & 10 & 545 \\ 
			& Time(s) & 0.01 & 0.01 & 0.03 & 0.01 & 0.02 & 0.08 & 0.02 & 0.02 & 0.02 & 0.01 & 0.02 & 0.50 \\ \hline 
			& Outer iterations & 5 & 6 & 5 & 6 & 5 & 5 &7 & 7 & 7 & 7 & 7 & 7 \\ 
			$10^{-4}$  & Inner iterations & 46 & 20 & 106 & 21 & 85 & 199 & 17 & 20 & 240 & 14 & 133 & 1234 \\ 
			& Time(s) & 0.05 & 0.02 & 0.11 & 0.03 & 0.08 & 0.18 & 0.02 & 0.03 & 0.23 & 0.02 & 0.13 & 1.10 \\ \hline 
			& Outer iterations & 6 & 11 & 6 & 8 & 6 & 6 &8 & 13 & 8 & 8 & 8 & 8 \\ 
			$10^{-5}$  & Inner iterations & 68 & 41 & 136 & 45 & 110 & 279 & 38 & 32 & 289 & 19 & 217 & 1868 \\ 
			& Time(s) & 0.07 & 0.05 & 0.14 & 0.05 & 0.11 & 0.25 & 0.04 & 0.05 & 0.27 & 0.02 & 0.21 & 1.66 \\ \hline 
			& Outer iterations & 7 & 40 & 7 & 10 & 7 & 7 &10 & 163 & 10 & 10 & 9 & 9 \\ 
			$10^{-6}$  & Inner iterations & 87 & 148 & 202 & 86 & 162 & 411 & 71 & 332 & 439 & 51 & 307 & 2507 \\ 
			& Time(s) & 0.08 & 0.19 & 0.20 & 0.09 & 0.16 & 0.37 & 0.08 & 0.53 & 0.41 & 0.05 & 0.29 & 2.22 \\ \hline 
			& Outer iterations & 8 & 94 & 7 & 11 & 7 & 7 &12 & 1098 & 11 & 12 & 11 & 11 \\ 
			$10^{-7}$  & Inner iterations & 113 & 362 & 202 & 115 & 162 & 411 & 158 & 2202 & 532 & 85 & 649 & 3468 \\ 
			& Time(s) & 0.11 & 0.47 & 0.20 & 0.12 & 0.16 & 0.37 & 0.16 & 3.51 & 0.49 & 0.09 & 0.59 & 3.06 \\ \hline 
			& Outer iterations & 9 & 118 & 8 & 13 & 8 & 8 &13 & 4738 & 12 & 22 & 11 & 11 \\ 
			$10^{-8}$  & Inner iterations & 138 & 441 & 300 & 149 & 252 & 587 & 234 & 9482 & 755 & 284 & 649 & 3468 \\ 
			& Time(s) & 0.13 & 0.58 & 0.29 & 0.15 & 0.24 & 0.52 & 0.24 & 15.13 & 0.69 & 0.28 & 0.59 & 3.06 \\ \hline 
\end{tabular}} \end{table}

\begin{table}
	\centering
	\caption{Numerical comparison on rcv1\_train.binary dataset with $\lambda = 10^{-4}$ and $\lambda = 10^{-6}$} \label{table_rcv1}
	\resizebox{\textwidth}{!}{
		\begin{tabular}{l  l || r r r r r r || c  c  c  c  c c }
			\hline 
			& & \multicolumn{6}{c||}{$\lambda = 10^{-4}$} & \multicolumn{6}{c}{$\lambda = 10^{-6}$} \\ \cline{3-14} 
			$\mathtt{TOL}$ & Solver & PNOPT & IRPN & IRPN & Algorithm \ref{Newton} & Algorithm \ref{Newton}& Algorithm \ref{Newton} & PNOPT & IRPN & IRPN & Algorithm \ref{Newton} & Algorithm \ref{Newton} & Algorithm \ref{Newton} \\ 
			& & & $\rho = 0$ & $\rho = 0.5$ & $\rho = 0.1$ & $\rho = 0.5$ & $\rho = 1$ & & $\rho = 0$ & $\rho = 0.5$ & $\rho = 0.1$ & $\rho = 0.5$ & $\rho = 1$ \\ \Xhline{1.2pt} \rule{0pt}{10pt}& Outer iterations & 3 & 3 & 3 & 3 & 3 & 3 &4 & 4 & 4 & 4 & 4 & 4 \\ 
			$10^{-3}$  & Inner iterations & 13 & 10 & 30 & 7 & 23 & 96 & 13 & 12 & 69 & 8 & 45 & 511 \\ 
			& Time(s) & 0.12 & 0.12 & 0.24 & 0.08 & 0.19 & 0.68 & 0.13 & 0.16 & 0.55 & 0.09 & 0.37 & 3.77 \\ \hline 
			& Outer iterations & 5 & 5 & 4 & 5 & 4 & 4 &6 & 6 & 6 & 8 & 6 & 6 \\ 
			$10^{-4}$  & Inner iterations & 52 & 22 & 43 & 21 & 42 & 190 & 36 & 16 & 199 & 20 & 208 & 1851 \\ 
			& Time(s) & 0.39 & 0.23 & 0.35 & 0.19 & 0.33 & 1.33 & 0.31 & 0.21 & 1.50 & 0.22 & 1.52 & 13.09 \\ \hline 
			& Outer iterations & 6 & 8 & 5 & 7 & 5 & 5 &9 & 24 & 9 & 11 & 9 & 9 \\ 
			$10^{-5}$  & Inner iterations & 95 & 81 & 84 & 64 & 96 & 326 & 167 & 83 & 1146 & 107 & 1183 & 5515 \\ 
			& Time(s) & 0.69 & 0.74 & 0.64 & 0.51 & 0.71 & 2.30 & 1.24 & 0.90 & 8.13 & 0.85 & 8.31 & 38.39 \\ \hline 
			& Outer iterations & 7 & 10 & 6 & 9 & 6 & 6 &11 & 125 & 10 & 13 & 10 & 10 \\ 
			$10^{-6}$  & Inner iterations & 139 & 121 & 132 & 121 & 180 & 506 & 501 & 362 & 2003 & 326 & 1917 & 13041 \\ 
			& Time(s) & 1.01 & 1.07 & 0.99 & 0.91 & 1.30 & 3.58 & 3.58 & 4.05 & 14.04 & 2.38 & 13.43 & 90.34 \\ \hline 
			& Outer iterations & 8 & 12 & 7 & 11 & 7 & 7 &12 & 386 & 12 & 15 & 12 & 11 \\ 
			$10^{-7}$  & Inner iterations & 192 & 161 & 256 & 191 & 303 & 765 & 806 & 1115 & 4154 & 908 & 5777 & 23041 \\ 
			& Time(s) & 1.38 & 1.42 & 1.89 & 1.42 & 2.15 & 5.35 & 5.69 & 12.44 & 28.93 & 6.78 & 40.11 & 159.03 \\ \hline 
			& Outer iterations & 9 & 14 & 7 & 12 & 7 & 7 &13 & 998 & 12 & 16 & 12 & 12 \\ 
			$10^{-8}$  & Inner iterations & 236 & 201 & 256 & 223 & 303 & 765 & 1259 & 2503 & 4154 & 1338 & 5777 & 33041 \\ 
			& Time(s) & 1.69 & 1.76 & 1.89 & 1.66 & 2.15 & 5.35 & 8.86 & 29.48 & 28.93 & 9.82 & 40.11 & 228.39 \\ \hline 
\end{tabular}} \end{table}%
We employ the two values $\{10^{-4}, 10^{-6}\}$ of
penalty parameter $\lambda$ for each dataset and the six levels
$\{10^{-2},10^{-3},10^{-4},10^{-5},10^{-6},10^{-7}\}$ of accuracy
$\mathtt{TOL}$ in the algorithms, and report the number of outer and
inner iterations along with CPU time, where the inner iterations
denote the total number of coordinate descent cycles of the
coordinate gradient descent method during implementation. In all of
the tests, line search procedures of each tested method provide the
unit step size. This may be because the zero point is a good initial
point for all tested problems. It can be observed from the numerical
results that our proposed proximal Newton-type
Algorithm~\ref{Newton} with $\rho = 0.1$ achieves the desired
accuracy with the least total iteration number and time in many
tested problems. Though in some tested problems, PNOPT achieves the
desired accuracy with the least total iteration number and time,
Algorithm~\ref{Newton} with $\rho = 0.1$ is still comparable with
PNOPT. Although there is no theoretical guarantee for PNOPT in
problems \eqref{logistic_regression} in the absence of the strong
convexity assumption, this algorithm happens to be efficient in
practice. It can also be seen that Algorithm~\ref{Newton} with $\rho
= 0.5$ and $1$ always achieves the desired accuracy with the least
outer iteration number. This supports the convergence rate result of
Theorem~\ref{quad} telling us that with larger values of $\rho$
Algorithm~\ref{Newton} achieves a higher order of the convergence
rate. However, larger $\rho$ causes $\eta_k$ in the inexact
condition \eqref{inexact_condition_1} to decrease faster, which
makes the inexact condition \eqref{inexact_condition_1} more
restrictive, and it will take more inner iterations for solving the
subproblem at each outer iteration. Hence, in practice, for the
total computation time we need to trade off between the outer
iteration number and the inner iteration number for solving each
subproblem. For the $l_1$ regularized logistic regression problem
\eqref{logistic_regression} tested in this section, and when we use
the coordinate gradient descent method for solving the subproblem at
each iteration of Algorithm~\ref{Newton}, the value $\rho = 0.1$
seems to be a good choice for achieving an overall efficient
performance.

        {\bf Acknowledgements}. The authors are very grateful to
        two anonymous referees for their helpful suggestions and
        remarks that allowed us to significantly improve the
        original presentation. The alphabetical order of the authors
indicates their equal contributions to the paper.

\end{document}